\documentclass[1996/10/24]{amsart}
\usepackage{color}
\usepackage{amssymb,stmaryrd}
\usepackage{graphicx}

\usepackage{color}
\definecolor{black}{rgb}{0,0,0}

\definecolor{red}{rgb}{1,0,0}

\definecolor{blue}{rgb}{0,0,1}

\newcommand{\Ct}{{\mathcal T}}

\newcommand{\Cn}{{\mathcal N}}

\def\bff{\textit{\textbf{f}}}

\def\bfg{\textit{\textbf{g}}}
\def\bfn{\textit{\textbf{n}}}

\def\bfv{\textit{\textbf{v}}}

\def\bfI{\textit{\textbf{I}}}

\def\bfV{\textit{\textbf{V}}}

\usepackage{url}

\newtheorem{thm}{Theorem}[section]

\newtheorem{lem}[thm]{Lemma}
\theoremstyle{remark}

\newcommand{\bu}{{\bf u}}
\newcommand{\bv}{{\bf v}}
\newcommand{\bbf}{{\bf f}}
\newcommand{\bx}{{\bf x}}
\newcommand{\dx}{~d{\bf x}}
\newtheorem{algorithm}{\bf Algorithm}

\newtheorem{remark}{Remark}[section]
\numberwithin{equation}{section}
\numberwithin{figure}{section}
\begin{document}

\title[Threshold dynamics method for topology optimization for fluids]{An efficient threshold dynamics method for topology optimization for fluids}

\author{Huangxin Chen}
\address{School of Mathematical Sciences and Fujian Provincial Key Laboratory on Mathematical Modeling and High Performance Scientific Computing, Xiamen University, Fujian, 361005, China}
\email{chx@xmu.edu.cn}

\author{Haitao Leng}
\address{Department of Mathematics, The Hong Kong University of Science and Technology, Clear Water Bay, Kowloon, Hong Kong, China}
\email{mahtleng@ust.hk}

\author{Dong Wang}
\address{Department of Mathematics, University of Utah, Salt Lake City, Utah, USA}
\email{dwang@math.utah.edu}

\author{Xiao-Ping Wang}
\address{Department of Mathematics, The Hong Kong University of Science and Technology, Clear Water Bay, Kowloon, Hong Kong, China}
\email{mawang@ust.hk}

\thanks{ This research was supported in part by the Hong Kong Research Grants Council (GRF grants 16302715, 16324416, 16303318, and NSFC-RGC joint research grant N-HKUST620/15).  H. Chen and D. Wang acknowledge the hospitality of the Department of Mathematics at the Hong Kong University of Science and Technology during their visit.The work of H. Chen was supported by the NSF of China (Grant No. 11771363, 91630204, 51661135011), the Fundamental Research Funds for the Central Universities (Grant No. 20720180003), and the Program for Prominent Young Talents in Fujian Province University.
}



\date{}

\begin{abstract}
We propose an efficient  threshold dynamics method for topology optimization for fluids modeled with the Stokes equation. The proposed algorithm is based on minimization of an objective energy function that consists of the dissipation power in the fluid and the perimeter approximated by nonlocal energy, subject to a fluid volume constraint and the incompressibility condition. We show that the minimization problem can be solved with an iterative scheme in which the Stokes equation is approximated by a Brinkman equation. The indicator functions of the fluid-solid regions are then updated according to simple convolutions followed by a thresholding step. We demonstrate mathematically that the iterative algorithm has the total energy decaying property.
The proposed algorithm is simple and easy to implement.  A simple adaptive time strategy is also used to accelerate the convergence of the iteration. Extensive numerical experiments in both two and three dimensions show that the proposed iteration algorithm converges in much fewer iterations and is more efficient than many existing methods. In addition, the numerical results show that the algorithm is very robust and insensitive to the initial guess and the parameters in the model.
\end{abstract}

\maketitle

\section{Introduction}
Topology optimization in fluid mechanics has become a significant problem due to its application in many industrial problems such as the optimization of transport vehicles and biomechanical structure. The process of topology optimization allows the introduction of new boundaries as part of the solution and is thus more flexible than shape optimization, which requires that the topology be predetermined. The method of topology optimization was originally developed for the optimal design in structural mechanics \cite{BK1988,BS2003} and has been applied in a variety of physical fields such as acoustics, electromagnetics, fluid flow, and thermal problems \cite{BS2003,Borrvall2003,Duhring2008,Sigmund2008,van2014numerical,Dbouk2017}. Topology optimization was first applied to fluid mechanics by Borrvall and Petersson \cite{Borrvall2003} by adopting the concept of density methods to Stokes flows. In \cite{Borrvall2003}, the domain with fluid-solid regions was treated as the porous medium, the Brinkman flow was introduced to obtain a well-posed problem to minimize the total dissipation power, and the discrete optimization problem was further solved with the method of moving asymptotes (MMA) \cite{Svanberg1987} to obtain the optimal designed regions for fluids and solids. Topology optimization in fluid mechanics has since been extended to the Darcy-Stokes flow model \cite{Guest2006,Wiker2007}, Navier-Stokes flow \cite{Gersborg_Hansen2005,Olesen2006,Zhou2008,Deng2011,garcke2015numerical,Villanueva2017},  and non-Newtonian flow \cite{Pingen2010}, and it has also been applied in the design of more complicated fluidic devices \cite{Andreasen2009,Okkels2005,Okkels2007}.

Several successful methods have also been recently introduced to improve the performance of topology optimization in fluid mechanics. For instance, the level set method was applied to fluidic topology optimization (cf. \cite{Zhou2008,Challis2009,Villanueva2017} and the references therein), and the fluid-solid interface is described by the zero-level set of a level set function. In \cite{Villanueva2017}, the authors further studied the fluidic topology optimization framework by combining the level set method and the extended finite-element method. Phase field-based topology optimization for fluids was considered in \cite{garcke2015numerical}, in which the gradient flow method was used to find the optimal topology. Among these methods, a critical step is to update the fluid-solid regions by solving the Hamilton-Jacobi equations in the level set method \cite{Zhou2008}, by solving a parameter optimization problem via a nonlinear programming method \cite{Villanueva2017}, or by solving the Cahn-Hilliard or Allen-Cahn system via the phase field approach \cite{garcke2015numerical}.

The  threshold dynamics method developed by Merriman, Bence, and Osher (MBO) \cite{merriman1992diffusion,MBO1993,merriman1994motion} is an efficient method for simulation of the motion of the interface driven by the mean curvature.
To be more precise, let $D \subset \mathbb{R}^d$ be a domain whose boundary
$\Gamma= \partial D$ is to be evolved via motion by mean curvature.
The MBO method is an iterative method, and at each time step, it generates a new interface, $\Gamma_{\text{new}}$ (or equivalently, $D_{\text{new}}$) via the following two steps:

\smallskip

\noindent {\bf Step 1.} Solve the Cauchy initial value problem for the heat diffusion equation until time $t = \tau$,
\begin{align*}
& u_t = \Delta u , \\
& u(t=0,\cdot) = \chi_{D},
\end{align*}
where $\chi_D$ is the indicator function of domain $D$. Let $\tilde u(x) = u(\tau,x)$.

\smallskip

\noindent  {\bf Step 2.} Obtain a new domain $D_{\text{new}}$ with boundary $\Gamma_{\text{new}} = \partial D_{\text{new}}$  by
\begin{align*}
D_{\text{new}} = \left\{ x\colon \tilde u (x) \geq \frac{1}{2} \right\}.
\end{align*}

The MBO method has been shown to converge to continuous motion by mean curvature
\cite{barles1995simple,chambolle2006convergence,evans1993convergence,swartz2017convergence}. Esedoglu and Otto gave a variational formulation for the original MBO scheme and  successfully generalized this type of method to multiphase problems with arbitrary surface tensions \cite{esedoglu2015threshold}. The method has attracted considerable attention due to its simplicity and unconditional stability. It has since been extended to deal with many other applications, including the problem of area-preserving or volume-preserving interface motion \cite{ruuth2003simple},
image processing \cite{wang2016efficient,esedog2006threshold,merkurjev2013mbo},
problems of  anisotropic interface motion \cite{merriman2000convolution,ruuth2001convolution,bonnetier2010consistency,elsey2016threshold},
the wetting problem on solid surfaces \cite{xu2016efficient},
the generation of quadrilateral meshes \cite{Viertel2017},
graph partitioning and data clustering  \cite{Gennip2013},
and auction dynamics \cite{jacobsauction}.
Various algorithms and rigorous error analysis have been introduced to refine and extend the original MBO method and related methods for these problems (see, {\it e.g.}, \cite{esedoglu2008threshold,
ishii2005optimal,merriman1994motion,
ruuth1998diffusion,ruuth1998efficient,wang2001gauss}).
Adaptive methods have also been used to accelerate this type of method \cite{jiang2016nufft} based on nonuniform fast Fourier transform. Laux et al. \cite{laux2016convergence,laux2016convergence2} rigorously proved the convergence of the method proposed by \cite{esedoglu2015threshold}, and a generalized manifold-valued threshold dynamics method was developed by \cite{osting2017generalized,wang2018diffusion,osting2018diffusion}.

 In this paper, we introduce an efficient and simple strategy based on the threshold dynamics method to update the topology of fluid-solid regions.
 In our approach, the total energy consists of the dissipation power in the fluid and the perimeter regularization and is subject to a fluid volume constraint and an incompressibility condition. The perimeter term is based on convolution of the heat kernel with the characteristic functions of regions. Based on minimization of an approximate total energy, an efficient threshold dynamics method is derived for topology optimization for fluids. The porous medium approach is used in our algorithm, and we introduce
 the Brinkman equation, which ``interpolates'' between the Stokes equation in the flow region and some Darcy flow through a porous medium (a weakened nonfluid region).
We then solve the Brinkman equation for the whole domain by the standard mixed finite-element method and update the fluid-solid regions by convolution and with a simple thresholding step. In particular, the convolutions can be efficiently computed on a uniform grid by fast Fourier transform (FFT) with the an optimal complexity of $O(N \log N)$.
The proposed algorithm is very simple and easy to implement.  Extensive numerical results show that the proposed algorithm converges at many fewer iterations than the method given by \cite{Borrvall2003},  which indicates the high efficiency of the proposed algorithm. In addition, the numerical results show that the algorithm is very robust and insensitive to the initial guess and the parameters.  We  also show that the method has  the total energy decaying property.

The paper is organized as follows. In Section \ref{mainres}, we show the mathematical model. In Section \ref{sec:derivation}, we introduce an approximate energy to the total energy and derive an efficient threshold dynamics method. The unconditional stability of the threshold dynamics method ($i.e.$, the energy decaying property) is proved in Section \ref{sec:stability}. We discuss the numerical implementation in Section \ref{sec:implementation} and verify the efficiency and the energy decaying property of the algorithm in Section \ref{sec:example}. We make conclusions, and discuss some ideas for future work in Section \ref{sec:conclusion}.

\section{Mathematical model}\label{mainres}
In this section, we consider the mathematical model for topology optimization for fluids in Stokes flow. Denote $\Omega \in \mathbb{R}^d$ $(d=2,3)$ as the computational domain, which is fixed throughout optimization, and assume that $\Omega$ is a bounded Lipschitz domain with an outer unit normal ${\bf n}$ such that $\mathbb{R}^d \setminus \overline{\Omega}$ is connected. Furthermore, we denote $\Omega_0 \subset \Omega$ as the domain of the fluid, which is a Caccioppoli set\footnote{In mathematics, a Caccioppoli set is a set whose boundary is measurable and has a (at least locally) finite measure. \url{https://en.wikipedia.org/wiki/Caccioppoli_set}} and $\Omega \setminus \Omega_0  \in \Omega$ as the solid domain. Throughout the paper, we use the standard notations and definitions for Sobolev spaces (cf. \cite{Adams1975}). Our goal is to determine an optimal shape of  $\Omega_0$ that minimizes the following objective functional consisting of  the total potential power and a perimeter regularization term,
\begin{align}
\min_{(\Omega_0,\bu)} J_0(\Omega_0,\bu) = \int_{\Omega}\left( \frac{\mu}{2}|D \bu|^2 - \bu \cdot \bbf\right)\dx + \gamma |\Gamma| \label{ener_ori}
\end{align}
subject to
\begin{subequations}\label{cosSharp}
\begin{align}
\nabla\cdot\bu= 0, &  \ \ \textrm{in}  \ \ \Omega, \\
\nabla p - \nabla \cdot (\mu \nabla \bu)  = \bbf, &   \ \ \textrm{in}  \ \ \Omega_0, \label{cosfluid}\\
\bu=0,& \  \ {\rm in} \ \Omega \setminus \Omega_0, \label{cossolid}\\
\bu|_{\partial \Omega} = \bu_D, & \ \ \textrm{on} \ \ \partial \Omega,  \\
|\Omega_0| = \ \beta |\Omega| &  \ \textrm{with a fixed parameter}  \  \beta \in (0,1).
\end{align}
\end{subequations}
Here,  $\bu: \Omega \rightarrow  \mathbb{R}^d$, $D \bu$ is the distributional derivative of $\bu$, $\mu$ is the dynamic viscosity of the fluid, $p$ is the pressure, $\bu_D: \partial \Omega \rightarrow \mathbb{R}^d$  is a given function, $\bbf: \Omega \rightarrow  \mathbb{R}^d$ is a given external force, $ |\Gamma| $ is the perimeter of the boundary of $\Gamma=\partial \Omega_0$, and $\gamma>0$ is a weighting parameter.

\section{Derivation of the algorithm}\label{sec:derivation}

In this section, we  develop an efficient threshold dynamics method for the topology optimization problem discussed in (\ref{ener_ori}) and (\ref{cosSharp}) for fluids in Stokes flow.
Note that the goal is to determine the optimal interface between liquid and solid that  minimizes functional  \eqref{ener_ori} subject to constraints \eqref{cosSharp}.
Motivated by the idea from the  threshold dynamics methods developed by \cite{esedoglu2015threshold}, \cite{xu2016efficient}, \cite{wang2016efficient}, we use the indicator functions for the fluid region and the solid region to implicitly represent the interface.

\subsection{Approximate energy}

Define an admissible set $\mathcal{B}$ as follows:
\begin{align}
\mathcal{B}:= &\{(v_1,v_2)\in BV(\Omega)\ |\ v_i(x)=\{0, 1\},  v_1(x)+v_2(x)=1 \ a.e. \ {\rm in} \ \Omega,\hbox { and } \int_{\Omega} v_1\dx=V_0  \},\label{def_b}
\end{align}
where $BV(\Omega)$ is the vector space of functions with bounded variation in $\Omega$, and $V_0$ is the fixed volume of the fluid region. We introduce $\chi_1(\bx)$ to denote the indicator function of the fluid region $\Omega_0$, $i.e.$,
\[\chi_1(\bx) :=
\begin{cases}
1,  & \textit{if}  \ \   \bx \in \Omega_0, \\
0,  & \textit{otherwise},
\end{cases}
\]
and $\chi_2(\bx)$ as the indicator function of $\Omega \setminus \Omega_0$, $i.e.$, $\chi_2(\bx) = 1- \chi_1(\bx)$. The interface $\Gamma$ is then implicitly represented by $\chi_1$ and $\chi_2$. Let $\chi = (\chi_1,\chi_2)$ and we have $\chi \in \mathcal{B}$. It is well known that  the perimeter of the interface $\Gamma$ can be approximated by,
\begin{align} \label{perimeter}
|\Gamma| \approx \sqrt{\frac{\pi}{\tau}}\int_{\Omega} \chi_1 G_\tau * \chi_2 \dx ,
\end{align}
where $G_\tau(\bx) = \dfrac{1}{(4\pi \tau)^{\frac{d}{2}}} \exp\left(-\dfrac{|\bx|^2}{4\tau}\right)$ is the Gaussian kernel (See \cite{esedoglu2015threshold}).

We solve  the optimization problem \eqref{cosSharp} by iteration. At each iteration, one must solve the Stokes equation in the fluid domain, which is changing in the iteration. It is more convenient numerically to use  the porous medium approach  as in \cite{garcke2015numerical,chen2014one}.  The idea is to ``interpolate"  between the Stokes equation in the fluid domain ($i.e.$, $\{\bx| \ \chi_1(\bx) =1\}$) and $\bu = 0 $ in the solid domain ($i.e.$, $\{\bx| \ \chi_2(\bx) =1\}$) by introducing an  additional penalization term,
\begin{subequations}\label{cos}
\begin{align}
\nabla\cdot\bu= 0, &  \ \ \textrm{in}  \ \ \Omega, \label{cos1} \\
\nabla p - \nabla \cdot (\mu \nabla \bu) + \alpha(\bx) \bu = \bbf, &   \ \ \textrm{in}  \ \ \Omega, \label{cos2} \\
\bu|_{\partial \Omega} = \bu_D, & \ \ \textrm{on} \ \ \partial \Omega. \label{cos3}
\end{align}
\end{subequations}
Here, $\alpha(\bx)$ is a smooth function that varies between $0$ and $\bar{\alpha}_\tau$  through a thin interface layer $\Gamma$, and  $\bar{\alpha}_\tau^{-1}$ is the permeability. In the current representation of the interface, we use the $0.5$ level set of  $\phi=G_{\tau}* \chi_2$ to approximate the position of the interface $\Gamma$.  It is well known that such $\phi$ is a smooth function between $[0,1]$ and admits a change from $0$ to $1$ in an $O(\sqrt{\tau})$ thin layer. Hence,  $\alpha$ is given by
\begin{align}\label{eq:alpha}
\alpha= \bar{\alpha}_\tau\phi  =\bar{\alpha}_\tau G_{\tau}* \chi_2 .
\end{align}
In the limiting model ($i.e.$, $\tau \searrow 0$), $\bar{\alpha}_\tau$ should be set as $+\infty$ to make the constraints $\{\bu=0 \ \ \textrm{in}  \ \ \Omega\setminus\Omega_0\}$ satisfy. Also, to ensure that the velocity vanishes outside the fluid domain when $\tau \searrow 0$, we add a penalty term $\frac{\bar{\alpha}_\tau }{2} G_{\tau}* \chi_2 |\bu|^2 $ to the objective functional. In subsequent calculations, for numerical consideration, we fix $\bar{\alpha}_\tau$ as a sufficiently large constant, $\bar{\alpha}$. In this porous media approach, the system (\ref{cos}) is solved for a fixed domain $\Omega$.

Finally, combining \eqref{ener_ori}, \eqref{perimeter}, \eqref{eq:alpha}, and the penalty term, we arrive at the following approximate objective functional
\begin{align} \label{eq:objapprox2}
J^\tau(\chi,\bu) =& \int_{\Omega} \left( \frac{\mu}{2} |D \bu|^2+ \frac{\bar{\alpha}}{2}|\bu|^2G_{\tau}* \chi_2 -\bu \cdot \bbf+ \gamma \sqrt{\frac{\pi}{\tau}} \chi_1 G_\tau * \chi_2 \right) \dx .
\end{align}
\begin{remark}
For simplicity, we use the same $\tau$ in the second and the fourth terms of the above approximate energy. Indeed, one can also use different values of $\tau$ in the two terms and the property of the algorithm  will be similar.
\end{remark}

Now, we consider the following approximate formulation of the problem by
\begin{align}\label{opt}
\min_{(\chi,\bu)} J^\tau(\chi,\bu), \ \textrm{subject to} \   \chi =(\chi_1, \chi_2) \in \mathcal{B} \ \textrm{and} \  \bu \ \textrm{satisfy} \  \eqref{cos}.
\end{align}
In the following, we give  the derivation of the  threshold dynamics scheme to solve \eqref{opt}.
\subsection{Derivation of the scheme.}
In this section, we use a coordinate descent algorithm to minimize the approximate energy (\ref{eq:objapprox2}) with constraints \eqref{cos}. A similar idea has been applied in the design of a threshold dynamics method of image segmentation \cite{wang2016efficient}. Given an initial guess $\chi^0 = (\chi_1^0,\chi_2^0)$, we compute a series of minimizers
\begin{align*}
\bu^0, \chi^1, \bu^1, \chi^2, \cdots, \bu^{k}, \chi^{k+1}, \cdots
\end{align*}
such that
\begin{align}
\bu^{k}=&\arg \min_{\bu\in \mathcal{S}}J ^\tau(\chi^k,\bu), \label{Eq:MinApproEnergySplit1}\\
\chi^{k+1}=&\arg \min_{\chi \in  \mathcal{B}}J ^\tau(\chi,\bu^{k}), \label{Eq:MinApproEnergySplit2}
\end{align}
for $k=0,1,2,\cdots .$ Here, the admissible set $\mathcal{S}$ is defined as
\[\mathcal{S}:=\left\{\bu \in H_{\bu_D}^1(\Omega,\mathbb{R}^d)\ |\  \nabla \cdot \bu = 0 \right\}\]
 where $H_{\bu_D}^1(\Omega,\mathbb{R}^d) = \{ \bu \in H^1(\Omega,\mathbb{R}^d)  \  | \ \bu |_{\partial \Omega} = \bu_D \}$, and $\mathcal{B}$ is defined in (\ref{def_b}).

Given  the $k$-th iteration $\chi^k$, we first  solve \eqref{Eq:MinApproEnergySplit1} to get the $\bu^{k}$.
It is easy to see that the constraint minimization problem is equivalent to the following
\[\bu^{k}= \arg \min_{\bu\in H_{\bu_D}^1(\Omega,\mathbb{R}^d)}J ^\tau(\chi^k,\bu)+\int_{\Omega} p \nabla \cdot \bu \dx\]
with $p$ as a Lagrangian multiplier.  Variation of the above functional leads to the following Brinkman equation.
That is, $\bu^{k}$ can be obtained by solving
\begin{align}\label{eq:stokes}
\begin{cases}
\nabla \cdot \bu = 0, \quad {\rm in}\ \Omega\\
\nabla  p  - \nabla \cdot (\mu \nabla \bu) +  \alpha(\chi^k)\bu  = \bbf,\quad {\rm in}\ \Omega  \\
\bu|_{\partial \Omega} = \bu_D
\end{cases}
\end{align}
where $\alpha(\chi^k) =\frac{\bar{\alpha}}{2}G_{\tau}* \chi^k_2  $.
Because $J ^\tau(\chi^k,\bu)$ is convex in $\bu$,  the solution  $(\bu^k,p^k)$ of (\ref{eq:stokes}) is a minimizer of $ J^\tau(\chi^k,\bu)$.
The following lemma shows the existence of $\bu$ for the system \eqref{eq:stokes} for a given $\chi \in \mathcal{B}$.
\begin{lem}[\cite{garcke2015numerical,GP1986}] \label{lem:1}
For every $\chi \in \mathcal{B}$, some $\bu \in H_{\bu_D}^1(\Omega, \mathbb{R}^d)$ exist that satisfy $\nabla\cdot \bu = 0$ such that
\begin{align}\label{eq:solvability}
\int_{\Omega} \mu \nabla \bu \cdot \nabla \bv +  \alpha(\chi) \bu \cdot \bv \ \dx = \int_{\Omega} \bbf \cdot \bv\dx, \ \ \forall \, \bv\in \bfV ,
\end{align}
where $\bfV:=\{ \bfv \in H^1_0(\Omega,\mathbb{R}^d)  \   |  \ \nabla \cdot \bfv = 0\}$.
\end{lem}

Given $\bu^{k}$, we now rewrite the objective functional $J^\tau(\chi,\bu)$ into $\tilde{J}^{\tau,k}(\chi)$ as follows:
\begin{align}
\tilde{J}^{\tau,k}(\chi):= J^\tau(\chi,\bu^{k}) = \int_{\Omega} \frac{\bar{\alpha}}{2}\chi_2 G_\tau* |\bu^k|^2 \dx +  \gamma \sqrt{\frac{\pi}{\tau}} \int_{\Omega}\chi_1 G_\tau * \chi_2 \dx + \int_{\Omega} \dfrac{\mu}{2} |D \bu^k|^2 -\bu^k \cdot \bbf \dx.
\end{align}
The next step is to find  $\chi^{k+1}$ such that
\begin{align} \label{Eq:MinApprochi}
\chi^{k+1}=&\arg\min\limits_{\chi \in \mathcal{B}} \tilde{J}^{\tau,k}(\chi).
\end{align}
It is the minimization of a concave functional on a nonconvex admissible set $ \mathcal{B}$.  However, we can relax it to a problem defined on a convex admissible set by finding $r^{k+1}$ such that
\begin{align}
r^{k+1}=&\arg\min\limits_{r \in \mathcal{H}} \tilde{J}^{\tau,k}(r) , \label{Eq:RelaxMinApprochi}
\end{align}
where $\mathcal{H}$ is the convex hull of $\mathcal{B}$ defined as follows:
\begin{align}
\mathcal{H}:= &\{(v_1,v_2)\in BV(\Omega)\ |\ v_i(x)\in [0, 1], i=1,2, \hbox { and } v_1(x)+v_2(x)=1\ a.e. \ {\rm in} \ \Omega, \int_{\Omega} v_1\dx=V_0  \},
\end{align}

The following lemma shows that the relaxed problem (\ref{Eq:RelaxMinApprochi}) is equivalent to the original problem (\ref{Eq:MinApprochi}). Therefore, we can solve the relaxed problem (\ref{Eq:RelaxMinApprochi}) instead.

\begin{lem} \label{Lemma}
Let $\bu\in H_{\bu_D}^1(\Omega,\mathbb{R}^d)$ be a given function and $r = (r_1,r_2)$. Then
\begin{align}
\arg\min\limits_{r \in \mathcal{H}} \tilde{J}^{\tau,k}(r)=\arg\min\limits_{r \in \mathcal{B}} \tilde{J}^{\tau,k}(r).
\end{align}
\begin{proof}
Let  $\tilde{r}=(\tilde{r}_1,\tilde{r}_2)\in\mathcal{H}$ be a minimizer of the functional
$\tilde{J}^{\tau,k}(r)$ on $\mathcal{H}$.
Because $\mathcal{B}\subset\mathcal{H}$, we have
\begin{align*}
 \tilde{J}^{\tau,k}(\tilde{r}) &=\min_{r\in\mathcal{H}}\tilde{J}^{\tau,k}(r)\leq \min_{r\in\mathcal{B}} \tilde{J}^{\tau,k}(r) .
\end{align*}
Therefore, we need only prove  that $\tilde{r}\in\mathcal{B}$.

We prove  by contradiction.  If $\tilde{r}\not\in\mathcal{B}$, there is a set $A\in\Omega$
 and a constant $0<C_0<\frac12$, such that $|A|>0$ and
$$
0<C_0<\tilde r_1(\bx),\tilde r_2(\bx)<1-C_0, \ \ \ \hbox{for all } \bx\in A.
$$
We divide $A$ into two sets $A=A_1\cup A_2$ such that $A_1 \cap A_2=\emptyset$ and $|A_1|=|A_2|=|A|/2$.
Denote $r^t= (r_1^t,r_2^t)$ where $r_1^t=\tilde{r}_1+t \chi_{A_1}-t\chi_{A_2}$ and
$r_2^t=\tilde{r}_2-t \chi_{A_1}+t\chi_{A_2}$ with $\chi_{A_1}$ and $\chi_{A_2}$ being the indicator functions of the domain $A_1$ and $A_2$, respectively.
When $0<t<C_0$, we have $0<r_1^t, r_2^t<1$ and
\begin{equation*}
r_1^t+r_2^t= \tilde{r}_1+\tilde{r}_2=1,\ \hbox{and } \int_{\Omega} r_1^t\dx=\int_{\Omega}\tilde{r}_1\dx=V_0.
\end{equation*}
This implies that $r^t \in \mathcal{H}$.
Furthermore, direct computations give,
\begin{align*}
\frac{d^2}{dt^2}\tilde{J}^{\tau,k}(r)
&=2\gamma\frac{\sqrt{\pi}}{\sqrt{\tau}}\int_{{\Omega}}\frac{d}{dt} r_1^t G_{\tau}*\frac{d}{dt} r^t_2\dx\\
&=2\gamma\frac{\sqrt{\pi}}{\sqrt{\tau}}\int_{{\Omega}}(\chi_{A_1}-\chi_{A_2}) G_\tau*(\chi_{A_2}-\chi_{A_1})\dx\\
&=-2\gamma\frac{\sqrt{\pi}}{\sqrt{\tau}}\int_{{\Omega}}(\chi_{A_1}-\chi_{A_2}) G_\tau*(\chi_{A_1}-\chi_{A_2})\dx\\
& = -2\gamma\frac{\sqrt{\pi}}{\sqrt{\tau}}\int_{{\Omega}}\left(G_{\tau/2}*(\chi_{A_1}-\chi_{A_2}) \right)\left(G_{\tau/2}*(\chi_{A_1}-\chi_{A_2})\right)\dx \\
&\leq 0.
\end{align*}
The penultimate step comes from the fact that the heat kernel is a self-adjoint operator and forms a semigroup with various values of $\tau$.  From the above inequality, the functional is concave on the point $\tilde{r}$. Thus, $\tilde{r}$ cannot be a minimizer of the functional. This contradicts the assumption.
\end{proof}
\end{lem}
Now, we show that \eqref{Eq:RelaxMinApprochi} can be solved  with a simple threshold dynamics method. Because $\tilde{J}^{\tau,k}(r)$ is quadratic in $r$, we first linearize the energy $\tilde{J}^{\tau,k}(r)$ at $r^k$ by
\begin{align}\label{eq:Linearization}
\tilde{J}^{\tau,k}(r) \approx \tilde{J}^{\tau,k}(r^k)+\mathcal{L}^{\tau,k}_{r^k}(r - r^k),
\end{align}
where
\begin{align}\label{eq:Linearization1}
\mathcal{L}^{\tau,k}_{r^k}(r) = &   \int_{\Omega} \left( \gamma \sqrt{\frac{\pi}{\tau}} r_1 G_\tau * r_2 ^k+ \gamma \sqrt{\frac{\pi}{\tau}} r_2 G_\tau * r_1^k + r_2 \frac{\bar{\alpha}}{2} G_\tau * |\bu^k|^2 \right) \dx \\
=& \int_{\Omega}\left( r_1 \phi_1+r_2 \phi_2 \right) \dx. \nonumber
\end{align}
Here $\phi_1=\gamma \sqrt{\frac{\pi}{\tau}} G_\tau * r_2^k $ and $\phi_2=\frac{\bar{\alpha}}{2} G_\tau * |\bu^k|^2+\gamma \sqrt{\frac{\pi}{\tau}}G_\tau * r_1^k$. Then (\ref{Eq:RelaxMinApprochi}) can be approximately reformulated into
\begin{equation}
\label{Eq:MinApproEnergySplit3relaxlinear2}
\chi^{k+1} =\arg \min_{r \in \mathcal{H}} \mathcal{L}^{\tau,k}_{r^k}(r)=\arg \min_{r \in \mathcal{H}} \int_{\Omega} \left( r_1 \phi_1+r_2 \phi_2 \right) \dx.
\end{equation}
The following lemma, in particular, (\ref{321})  shows that \eqref{Eq:MinApproEnergySplit3relaxlinear2} can be solved in a pointwise manner by
\begin{align} \label{choicechi}
\begin{cases}
\chi_1^{k+1}(\bx) = 1 \hbox{ and } \chi_2^{k+1}(\bx) = 0,  \ \ \textit{if}  \ \ \phi_1(\bx)<\phi_2(\bx)+\delta, \\
\chi_1^{k+1}(\bx) = 0 \hbox{ and } \chi_2^{k+1} (\bx)= 1,  \ \ \textit{otherwise},
\end{cases}
\end{align}
where $\delta$ is chosen as a constant such that $\int_{\Omega} \chi_1^{k+1} \dx= V_0$.

\begin{lem}\label{lem:stability3}
Let $\phi_1=\gamma \sqrt{\frac{\pi}{\tau}} G_\tau * \chi_2^k $, $\phi_2=\frac{\bar{\alpha}}{2} G_\tau * |\bu|^2+\gamma \sqrt{\frac{\pi}{\tau}}G_\tau * \chi_1^k$ and
\begin{equation}
D_1^{k+1}=\{ \bx \in\Omega | \;\phi_1-\phi_2< \delta \}
\end{equation}
for some $\delta$ such that $|D_1^{k+1}|=V_0$. Then for $\chi^{k+1} = (\chi_1^{k+1},\chi_2^{k+1})$ with $\chi_1^{k+1} = \chi_{D_1^{k+1}}$ and $\chi_2^{k+1} = 1-\chi_1^{k+1}$, we have
\[ \mathcal{L}^{\tau,k}_{\chi^k}(\chi^{k+1})\leq \mathcal{L}^{\tau,k}_{\chi^k}(\chi^k) \]
for all $\tau>0$.
\end{lem}
\begin{proof}
Because $\mathcal{L}^{\tau,k}_{\chi^k}$ is a linear functional, we only need to prove that there holds
\begin{align} \label{321}
\mathcal{L}^{\tau,k}_{\chi^k}(\chi^{k+1})\leq \mathcal{L}^{\tau,k}_{\chi^k}(\chi)
\end{align}
for all $\chi = (\chi_1,\chi_2)\in\mathcal{B}$.

For each $(\chi_1,\chi_2)\in\mathcal{B}$, we know $\chi_1=\chi_{\hat{D}_1}$ and $\chi_2=\chi_{\hat{D}_2}$ for some open sets $\hat{D}_1$,
$\hat{D}_2$ in $\Omega$, such that $\hat D_1\cap \hat D_2=\emptyset$, $\hat D_1\cup \hat D_2=\Omega$ and $|\hat D_1|=V_0$.
Let $A_1=\hat D_1\setminus  D_1^{k+1}=D_2^{k+1}\setminus \hat D_2$ and $A_2=\hat D_2 \setminus D_2^{k+1}=D_1^{k+1}\setminus \hat D_1$.
We must have $|A_1|=|A_2|$ due to the volume conservation property. Because $A_1\subset D_2^{k+1}$, we have
$$\phi_1(\bx)-\phi_2(\bx)\geq \delta, \ \  \chi_1^{k+1}(\bx)-\chi_1(\bx)=-1, \ \ \ \ \forall \bx\in A_1.$$
Similarly, because $A_2\in D_1^{k+1}$, we have
$$\phi_1(\bx)-\phi_2(\bx)<\delta,\ \ \chi_1^{k+1}(\bx)-\chi_1(\bx)=1,  \ \ \ \ \forall \bx\in A_2.$$
Therefore, using $\chi_1^{k+1}-\chi_1+\chi_2^{k+1}-\chi_2 = 0$, we have
\begin{align*}
\mathcal{L}^{\tau,k}_{\chi^k}(\chi^{k+1})- \mathcal{L}^{\tau,k}_{\chi^k}(\chi)
=& \gamma \sqrt{\frac{\pi}{\tau}} \int_{\Omega} (\chi_1^{k+1}-\chi_1)\phi_1+(\chi_2^{k+1}-\chi_2)\phi_2 \dx \\
= &\gamma \sqrt{\frac{\pi}{\tau}} \int_{\Omega} (\chi_1^{k+1}-\chi_1)(\phi_1-\phi_2) \dx \\
= & \gamma \sqrt{\frac{\pi}{\tau}} \int_{\Omega} \left(\chi_{A_2}(\phi_1-\phi_2)- \chi_{A_1}(\phi_1-\phi_2)\right) \dx \\
\leq &\gamma \sqrt{\frac{\pi}{\tau}} \int_{\Omega}\left( \chi_{A_2}\delta- \chi_{A_1}\delta \right)\dx  = \gamma \sqrt{\frac{\pi}{\tau}} \delta (|A_2|-|A_1|) = 0.
\end{align*}
\end{proof}


To determine the value of $\delta$, one can treat $\int_{\Omega} \chi_1^{k+1} \dx - V_0$ as a function of $\delta$ and use an iteration method ($e.g.$, bisection method or Newton's method) to find the root of $\int_{\Omega} \chi_1^{k+1} \dx - V_0 = 0$. For the uniform discretization of $\Omega$, a more efficient method is the  quick-sort technique proposed in \cite{xu2016efficient}.  Assume we have a uniform discretization of $\Omega$ with grid size $h$, we can approximate $\int_{\Omega} \chi_1^{k+1} \dx $ by $M h^2$, we then sort the values of $ {\phi}_1- {\phi}_2$ in an ascending order and simply set $ \chi_1^{k+1} = 1$ on the first $M$ points and $ \chi_2^{k+1} = 1$ on the other points.

\begin{remark}\label{alg_rem}
In many implementations, one may solve Stokes equation on nonuniform grid points. To preserve the volume for the discretization on nonuniform grids, although the volume cannot be simply approximated by the number of grid points times the size of each cell, a similar technique can be applied.  One can still sort the values of $\phi_1-\phi_2$ in ascending order, save the index into $\mathcal{S}$, calculate the integrating weight at each grid point into $\mathcal{V}$, and set $V=0$ and $i=0$. Then, $\delta$ can be simply found by:
\[{\sl while} \ \ V<V_0; \ \ i \leftarrow i+1;  \ \ V = V+\mathcal{V}(\mathcal{S}(i));  \ \  {\sl end}; \ \  \delta = \phi_1(\mathcal{S}(i+1))- \phi_2(\mathcal{S}(i+1)).\]
\end{remark}
 Now, we are led to a threshold dynamics  algorithm for topology optimization problem (\ref{opt}) for fluids in Stokes flow in the following.

\begin{algorithm}\label{algorithm_main}

Discretize $\Omega$ uniformly into a grid $\Ct_h$ with grid size $h$ and set $M = V_0/h^d$.

\smallskip

{\bf Step 1. Input:} Set $\tau> 0$, $\bar{\alpha}>0$, $k=0$, a tolerance parameter $tol>0$ and give the initial guess $\chi^0 \in \mathcal{B}$.

\smallskip

{\bf Step 2. Iterative solution:}

{\bf 1. Update $\bu$.} Solve the Brinkman flow equations
\[
\begin{cases}
\nabla \cdot \bu = 0, \quad {\rm in} \ \Omega \\
\nabla  p  - \nabla \cdot (\mu \nabla \bu) +  \alpha(\chi^k) \bu  = \bbf,\quad {\rm in} \ \Omega \\
\bu|_{\partial \Omega} = \bu_D
\end{cases}
\]
by mixed finite-element method to get $\bu^k$, where $\alpha(\chi^k) =\bar{\alpha}G_{\tau}* \chi_2^k$.

{\bf 2. Update $\chi$.} Evaluate
\[\begin{cases} \phi_1 =\gamma \sqrt{\frac{\pi}{\tau}}G_\tau * \chi_2^k,\\
\phi_2 =\frac{ \bar{\alpha}}{2}G_{\tau}* |\bu|^2+ \gamma \sqrt{\frac{\pi}{\tau}}G_\tau * \chi_1^k.
\end{cases}\]
Sort the values of ${\phi}_1-{\phi}_2$ in an ascending order,  and set $ \chi_1^{k+1} = 1$ on the first $M$ points and $\chi_2^{k+1} =1$ on the other points.

{\bf 3. }Compute $e_{\chi}^k = \|\chi_1^{k+1}-\chi_1^k\|_2 $. If $e_{\chi}^k\leq tol$, stop the iteration and go to the output step. Otherwise,  let $k+1  \rightarrow k$ and continue the iteration.

\smallskip

{\bf Step 3. Output:} A function $\chi \in \mathcal{B}$ that approximately solves \eqref{opt}.
\label{a:MBO}
\end{algorithm}

\begin{remark}
We note that in the original MBO method, on one hand, the algorithm can be easily stuck when $\tau$ is very small because, in the discretized space, $\tau$ is so small that no point can switch from one phase to another ($i.e.$, $\chi_1$ changes from 0 to 1 or 1 to 0) at one iteration step. On the other hand, with a large $\tau$, the interface can easily move but creates large error.
Hence, we apply the adaptive in time technique \cite{xu2016efficient} in numerical experiments by modifying Algorithm \ref{a:MBO} into an adaptive algorithm by adjusting $\tau$ during the iterations. {Indeed, we set a threshold value $\tau_t$ and a given tolerance $e_t$, if $e_{\chi}^k \leq e_t$, let $\tau_{\rm new}=\eta \tau$ with $\eta \in (0,1)$ and update $\tau := \tau_{\rm new}$ in the next iteration unless $\tau \leq \tau_t$. Otherwise, $\tau$ will not be updated, and the iteration will continue with the same $\tau$.} We use this adaptive strategy for the choice of $\tau$ in the numerical experiments.
\end{remark}

\section{Stability Analysis}\label{sec:stability}
In this section, we prove the unconditional stability property of the proposed algorithm.  Specifically, for the series of minimizers
\begin{align*}
\bu^0, \chi^1, \bu^1, \chi^2, \cdots, \bu^{k}, \chi^{k+1}, \cdots,
\end{align*}
computed by Algorithm \ref{a:MBO}, we prove
\[J^\tau(\chi^{k+1},\bu^{k+1}) \leq J^\tau(\chi^k,\bu^k)\]
for all $\tau >0$.

We first introduce Lemma \ref{lem:stability2} which leads us to $J^\tau(\chi^{k+1},\bu^k) \leq J^\tau(\chi^k,\bu^k)$ for all $\tau>0$.

\begin{lem}\label{lem:stability2}
For a fixed $\bu^k$, let $\chi^{k+1}$ be the $k+1$-th iteration derived from Algorithm \ref{a:MBO}, we have
\[J^\tau(\chi^{k+1},\bu^k) \leq J^\tau(\chi^k,\bu^k)\] for all $\tau>0$.
\begin{proof}
From the  linearization of  $\tilde{J}^{\tau,k}(\chi^k)$ in \eqref{eq:Linearization}, we have
\begin{align*}
J^\tau(\chi^k,\bu^k) =& \mathcal{L}^{\tau,k}_{\chi^k}(\chi^k) -  \gamma \sqrt{\frac{\pi}{\tau}} \int_{\Omega}\chi_1^k G_\tau * \chi_2^k \dx + \int_{\Omega} \dfrac{\mu}{2} |D \bu^k|^2 -\bu^k \cdot \bbf \dx,\\
J^\tau(\chi^{k+1},\bu^k) = &\mathcal{L}^{\tau,k}_{\chi^k}(\chi^{k+1}) -  \gamma \sqrt{\frac{\pi}{\tau}} \int_{\Omega} \left(\chi_1^{k+1} G_\tau * \chi_2^k +\chi_2^{k+1} G_\tau * \chi_1^k - \chi_1^{k+1} G_\tau * \chi_2^{k+1} \right)  \dx \\
&+ \int_{\Omega} \dfrac{\mu}{2} |D \bu^k|^2  -\bu^k \cdot \bbf\dx \nonumber.
\end{align*}
Then, we calculate
\begin{align*}
 J^\tau(\chi^{k+1},\bu^k) -J^\tau(\chi^k,\bu^k) =& \mathcal{L}^{\tau,k}_{\chi^k}(\chi^{k+1})-  \mathcal{L}^{\tau,k}_{\chi^k}(\chi^k) + \gamma \sqrt{\frac{\pi}{\tau}} \int_{\Omega} (\chi_1^{k+1} -\chi^k_1)G_\tau *(\chi_2^{k+1}- \chi_2^k) \dx \\
 =& \mathcal{L}^{\tau,k}_{\chi^k}(\chi^{k+1})-  \mathcal{L}^{\tau,k}_{\chi^k}(\chi^k) - \gamma \sqrt{\frac{\pi}{\tau}} \int_{\Omega} (\chi_1^{k+1} -\chi^k_1)G_\tau *(\chi_1^{k+1}- \chi_1^k) \dx \\
 =& \mathcal{L}^{\tau,k}_{\chi^k}(\chi^{k+1})-  \mathcal{L}^{\tau,k}_{\chi^k}(\chi^k) - \gamma \sqrt{\frac{\pi}{\tau}} \int_{\Omega}\left( G_{\tau/2}*(\chi_1^{k+1} -\chi^k_1)\right)^2 \dx \\
 \leq & \mathcal{L}^{\tau,k}_{\chi^k}(\chi^{k+1})-  \mathcal{L}^{\tau,k}_{\chi^k}(\chi^k).
\end{align*}
Because we have $ \mathcal{L}^{\tau,k}_{\chi^k}(\chi^{k+1})-  \mathcal{L}^{\tau,k}_{\chi^k}(\chi^k) \leq 0$  from Lemma \ref{lem:stability3}, we are led to
\[J^\tau(\chi^{k+1},\bu^k) -J^\tau(\chi^k,\bu^k) \leq 0\]
for all $\tau>0$.
\end{proof}
\end{lem}

We are now led to the following theorem which proves the total energy decaying property
\begin{thm}\label{thm:unconditionalstability}
For the series of minimizers
\begin{align*}
\bu^0, \chi^1, \bu^1, \chi^2, \cdots, \bu^{k}, \chi^{k+1}, \cdots,
\end{align*}
calculated with Algorithm \ref{a:MBO}, we have
\begin{align}\label{energydecay}
J^\tau(\chi^{k+1},\bu^{k+1}) \leq J^\tau(\chi^k,\bu^k)
\end{align}
for all $\tau >0$.
\end{thm}
\begin{proof}
For all $\tau>0$, from (\ref{Eq:MinApproEnergySplit1}) , we have
\[J^\tau(\chi^{k+1},\bu^{k+1}) \leq J^\tau(\chi^{k+1},\bu^k).\]
From Lemma \ref{lem:stability2}, we have
\[J^\tau(\chi^{k+1},\bu^{k}) \leq J^\tau(\chi^k,\bu^k).\]
Thus, combining the above together gives  the stability estimate \eqref{energydecay}.
\end{proof}

\begin{remark}
We remark here that, as we proved, the energy is decaying for any given $\tau$. If $\tau$ changes from $\tau_1$ to $\tau_2$ at the $k^{th}$ iteration with $\tau_1>\tau_2$ in our adaptive in time  strategy, for example, $\chi^k$ is generated by $\tau_1$ and $\chi^{k+1}$ is generated by $\tau_2$. The energy is decaying in the sense that $J^{\tau_2}(\chi^{k+1},\bu^{k+1}) \leq J^{\tau_2}(\chi^{k},\bu^k)$ where the energy $J$ at two iterations $\chi^k$ and $\chi^{k+1}$ are approximated by the same $\tau_2$.
\end{remark}

\section{Numerical Implementation} \label{sec:implementation}

In this section, we illustrate the implementation of Algorithm \ref{a:MBO}, with a focus on Step 2. The Brinkman equations (\ref{cos1}-\ref{cos3}) are solved with the mixed finite-element method, and the Taylor-Hood finite-element space is used for discretization, which satisfies the discrete inf-sup condition \cite{GP1986}.

Let $\Ct_h$ be a uniform triangulation of the domain $\Omega$, and $\Cn_h$ is the set of all vertices of $\Ct_h$. For a given $\overline{\chi}_h = (\overline{\chi}^h_1,\overline{\chi}^h_2) \in \mathcal{B}_h$ where $\mathcal{B}_h$ is the discrete version of $\mathcal{B}$ defined on $\Cn_h$.  For the uniform regular triangulation of the domain, all values are evaluated on uniform quad grid points. Thus, we can use FFT for efficient evaluation of the discretized convolutions.

We introduce the Taylor-Hood finite-element space
\begin{align*}
\bfV_h & :  = \{ \bv\in H^1(\Omega,\mathbb{R}^d)  \ | \ \bv|_K \in [P_2(K)]^d, \ K \in \Ct_h \},\\
Q_h & : = \{ q \in L^2(\Omega,\mathbb{R}) \ | \ \int_\Omega q \ \dx = 0, \  q|_{K} \in P_1(K), \ K \in \Ct_h \}.
\end{align*}
Let $\bfV^{D}_h : = \{  \bv \in \bfV_h \ | \  \bv|_{\partial \Omega} =  \bu^h_D \}$, where $\bu^h_D $ is the a suitable approximation of the Dirichlet boundary condition $\bu_D$ on the boundary edges/faces of $\Ct_h$. For the solution of (\ref{cos1}-\ref{cos3}), find $(\bu_h,p_h) \in \bfV^{D}_h \times Q_h$ such that
\begin{align*}
- (p_h,\nabla \cdot \bv_h) + (\mu \nabla \bu_h, \nabla \bv_h) +( \alpha(\overline{\chi}_h) \bu_h,\bv_h) & = (\bbf,\bv_h), \quad \forall \ \bv_h  \ \in \bfV^0_h , \\
(\nabla\cdot\bu_h,q_h) & = 0, \quad \quad \quad  \ \forall \  q_h \in Q_h. \\
\end{align*}
The above bilinear form can be easily extended to the Brinkman equations both with Dirichlet boundary $\Gamma_D$ and Neumann boundary $\Gamma_N$, where $\Gamma_D \cap \Gamma_N = \emptyset, \Gamma_D \cup \Gamma_N = \partial \Omega$, and $(\mu \nabla \bu - p \bfI)\cdot \bfn |_{\Gamma_N} = \bfg$.

When $\bu_h$ is obtained, we proceed to use the FFT to evaluate $(\phi^h_1,\phi^h_2)$ on each node of $\Cn_h$ as follows:
\[\begin{cases} \phi^h_1 =\gamma \sqrt{\frac{\pi}{\tau}}G_\tau * \overline{\chi}_2^h,\\
\phi^h_2 =\frac{ \bar{\alpha}}{2}G_{\tau}* |\bu_h|^2+ \gamma \sqrt{\frac{\pi}{\tau}}G_\tau * \overline{\chi}_1^h.
\end{cases}\]

Following Algorithm \ref{a:MBO}, we can now use $(\phi^h_1,\phi^h_2)$ to update the indicator function $\chi_h$ by the approach stated in Algorithm \ref{a:MBO}.

\section{Numerical experiments}\label{sec:example}
In this section, we perform extensive numerical testing to demonstrate the efficiency of Algorithm \ref{a:MBO} with an adaptive strategy for the choice of $\tau$. We choose $\eta=0.5$ in the update of $\tau$. If no confusion is possible, we still denote by $\tau$ as its initialization in the following.

\subsection{Two dimensional results}
We firstly test the proposed algorithm for the two dimensional problems. For most of examples in this subsection, we assume that the Dirichlet boundary condition with a parabolic profile and the magnitude of the velocity is set as $|\bu_D| = \overline{g}(1-(2t/l)^2)$ with $t \in [-l/2,l/2]$, where $l$ is the length of the section of the boundary at which the inflow/outflow velocity is imposed. The direction of the inflow/outflow velocity is illustrated in the following examples.
\subsection*{Example 6.1}
The first example shown in Figure \ref{ex1f1} is the optimal design of a diffuser that was tested for topology optimization for fluids using MMA in \cite{Borrvall2003}. Here, we apply Algorithm \ref{algorithm_main} to obtain the optimal design of the diffuser. Let $\overline{g} = 1$ and $3$ for the inflow and outflow velocities, respectively. We set the fluid region fraction as $\beta=0.5$ and test the problem on a $128\times 128$ grid.

\begin{figure}[htbp]
\begin{center}
\includegraphics[width=9cm,height=5.4cm,clip,trim=0cm 6cm 0cm 6cm]{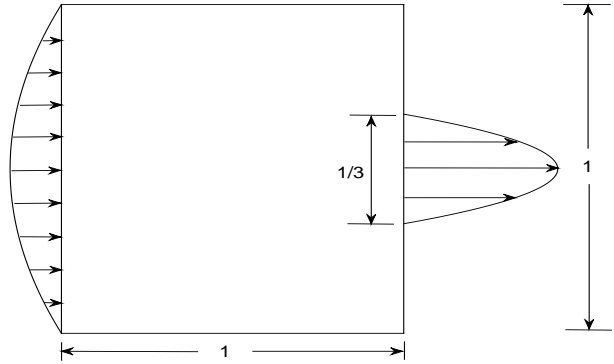}
\end{center}
\caption{\footnotesize (Example 6.1) Design domain for the diffuser example.}\label{ex1f1}
\end{figure}

\begin{figure}[htbp]
\begin{center}
\includegraphics[width=7cm,height=5.4cm,clip,trim=1cm 1cm 1cm 0cm]{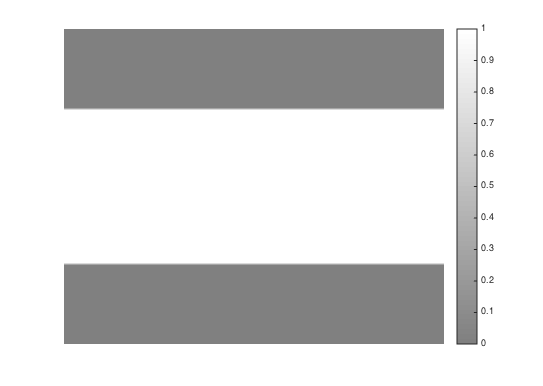}
\includegraphics[width=7cm,height=5.4cm,clip,trim=1cm 1cm 1cm 0cm]{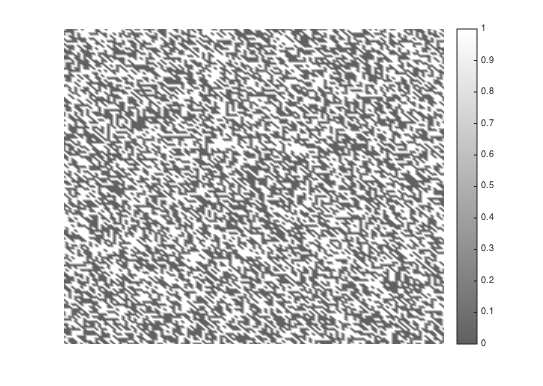}
\end{center}
\caption{\footnotesize (Example 6.1) Left (Case 1): Initial distribution of $\chi_1$. Right (Case 2): Initial distribution of $\chi_1$. }\label{ex1f2_dis}
\end{figure}

\begin{figure}[htbp]
\begin{center}
\includegraphics[width=7cm,height=5.4cm,clip,trim=2cm 1cm 1cm 1cm]{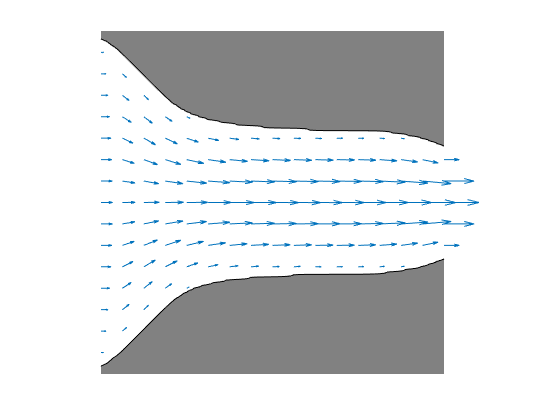}
\includegraphics[width=7cm,height=5.4cm,clip,trim=0.5cm 0cm 0cm 0.5cm]{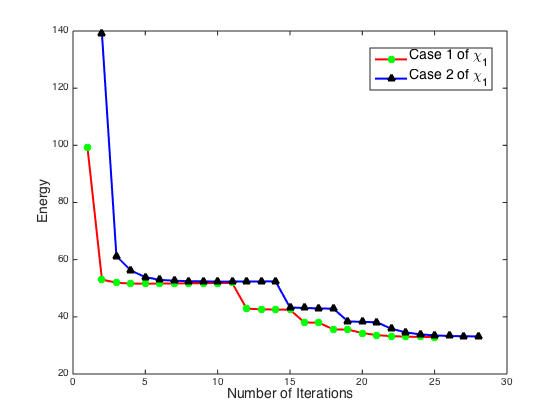}
\end{center}
\caption{\footnotesize (Example 6.1) Left: Optimal diffuser for the case $\bar{\alpha}=2.5\times10^4$ and the approximate velocity in the fluid region. Right: Plot of energy curves for two cases of distribution of $\chi_1$. In this case, the parameters are set as $\bar{\alpha}=2.5\times10^4$, $\tau = 0.01$, $\gamma = 0.1$.}\label{ex1f2_energ_vel_0}
\end{figure}

We first perform the simulations  with $\bar{\alpha}=2.5\times10^4$,  $\tau=0.01$,  $\gamma=0.1$ and with two types of initial distribution of $\chi_1$, as shown in Figure \ref{ex1f2_dis}; that is, the initial fluid region is restricted in the middle of the domain in the left graph of Figure \ref{ex1f2_dis} (Case 1), and the initial fluid region satisfies a random distribution in the right graph of Figure \ref{ex1f2_dis} (Case 2).   In both cases,  we always arrive at the same optimal design result shown in the left graph of Figure \ref{ex1f2_energ_vel_0}, which also shows the quiver plot of the approximate velocity in the fluid region. The optimal design result seems similar to the result obtained by MMA in \cite{Borrvall2003}. The energy decaying property can be observed in the right graph of Figure \ref{ex1f2_energ_vel_0} which shows the energy curves for the above two cases of the initial distribution of $\chi_1$.  The iteration converges in about $25$ steps in both cases.

\begin{figure}[htbp]
\begin{center}
\includegraphics[width=7cm,height=5.4cm,clip,trim=0cm 0cm 0cm 0.5cm]{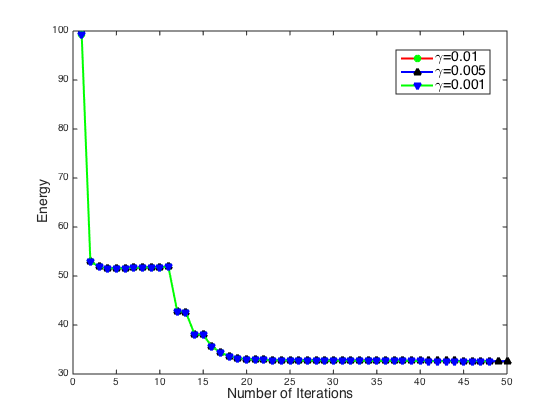}
\includegraphics[width=7cm,height=5.4cm,clip,trim=0cm 0cm 0cm 0.5cm]{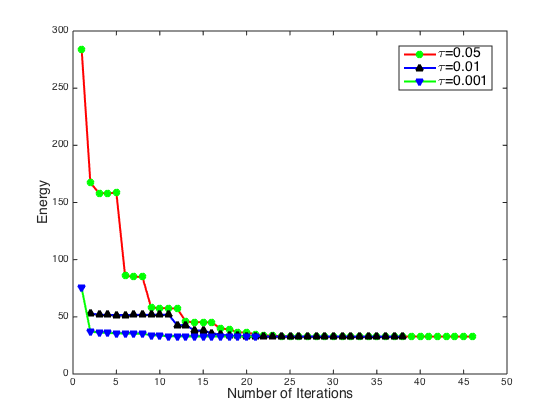}
\end{center}
\caption{\footnotesize (Example 6.1) Plot of energy curves for case 1 of distribution of $\chi_1$ with $\bar{\alpha}=2.5\times10^4$. Left: For fixed $\tau=0.01$, energy curves for the cases of $\gamma = 0.01,0.005,0.001$. Right: For fixed $\gamma=0.01$, energy curves for the cases of $\tau = 0.05,0.01,0.001$.}\label{ex1f2_diff_pram}
\end{figure}

Next, we test the case (initial fluid region of Case 1)  for various parameters.   We first fix  $\bar{\alpha}=2.5\times10^4$,  $\tau = 0.01$ and vary $\gamma = 0.01,0.005,0.001$.  We then test the cases for fixed $\gamma=0.001$ and various choices of $\tau = 0.05,0.01,0.001$. The optimal design of the diffuser is similar to the result in the left graph of Figure \ref{ex1f2_energ_vel_0}. Figure \ref{ex1f2_diff_pram} shows the energy decaying property for each of these cases.  In all cases, the iteration converges in fewer than $25$ steps.

\begin{figure}[htbp]
\begin{center}
\includegraphics[width=7cm,height=5.4cm,clip,trim=2cm 1cm 1cm 1cm]{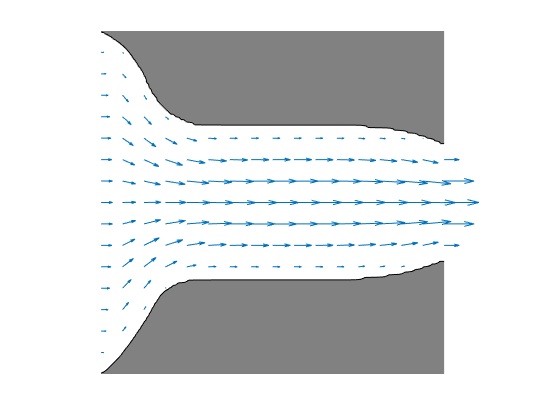}
\includegraphics[width=7cm,height=5.4cm,clip,trim=0.5cm 0cm 0cm 0.5cm]{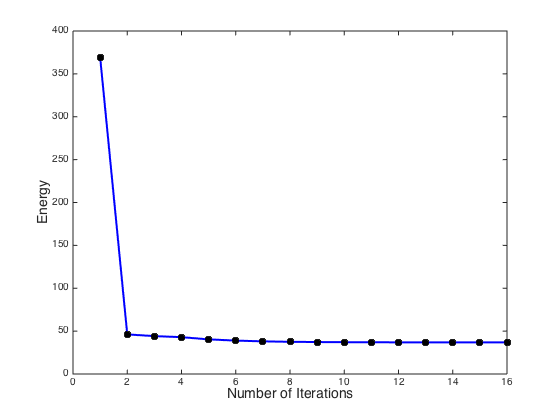}
\end{center}
\caption{\footnotesize (Example 6.1) Left: Associated optimal diffuser and approximate velocity in the fluid region. Right: Plot of energy curve for Case 1 of distribution of $\chi_1$. In this case, the parameters are set as $\bar{\alpha}=2.5\times10^5$, $\tau = 0.001$, $\gamma = 0.01$.}\label{ex1f2_energ_vel}
\end{figure}

In the next example,  we increase  $\bar{\alpha}=2.5\times10^5$.  Again, we use  the initial fluid region of Case 1 with  $\tau = 0.001,\gamma=0.01$. The optimal design of the diffuser and the approximate velocity in the fluid region are shown in the left graph of Figure \ref{ex1f2_energ_vel}.  It seems that the fluid region at the left boundary reaches top and bottom boundaries in this case.   The energy decaying property is also observed in  Figure \ref{ex1f2_energ_vel}.  The iteration converges even more quickly at about $10$ steps.

We also test the problem with the same inflow Dirichlet boundary condition as above, but we replace the outflow Dirichlet boundary condition with a homogeneous Neumann boundary. A similar optimal design of diffuser is then obtained as above for the cases of $\bar{\alpha}=2.5\times10^4$ and $\bar{\alpha}=2.5\times10^5$.

\subsection*{Example 6.2}
In this example, we test the double pipes problem shown in Figure \ref{ex2f1}. The inflow and outflow Dirichlet boundaries are located with centers $[0, 1/4], [0, 3/4], [1, 1/4], [1, 3/4]$, as shown in Figure \ref{ex2f1}. Let $\overline{g} =1 $ for the inflow and outflow velocities, respectively, and let the fluid region fraction be $\beta=1/3$. We test the problem with $\bar{\alpha}=2.5\times10^4$ on a $128\times 256$  grid for $d=0.5$ and on a $192\times 128$  grid  for $d=1.5$.

\begin{figure}[htbp]
\begin{center}
\includegraphics[width=8cm,height=5.4cm,clip,trim=0cm 0cm 0cm 0cm]{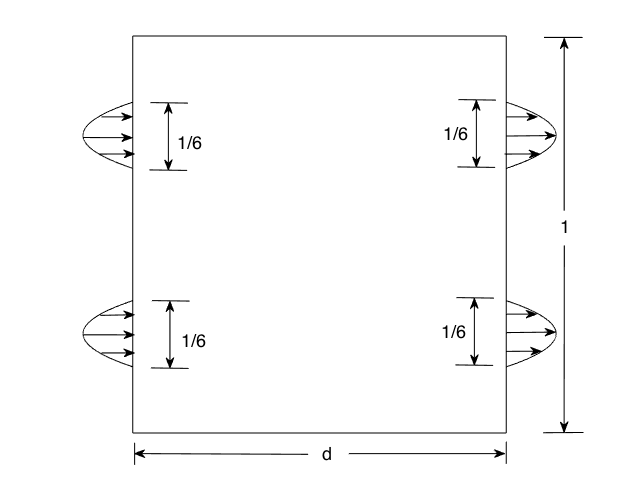}
\end{center}
\caption{\footnotesize (Example 6.2) Design domain for the double pipes example.}\label{ex2f1}
\end{figure}

\begin{figure}[htbp]
\begin{center}
\includegraphics[width=4cm,height=5.4cm,clip,trim=1.5cm 1cm 1.5cm 0cm]{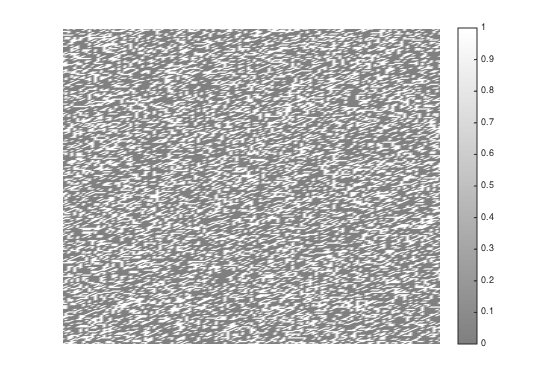}
\includegraphics[width=4.5cm,height=5.4cm,clip,trim=6cm 1cm 6cm 0cm]{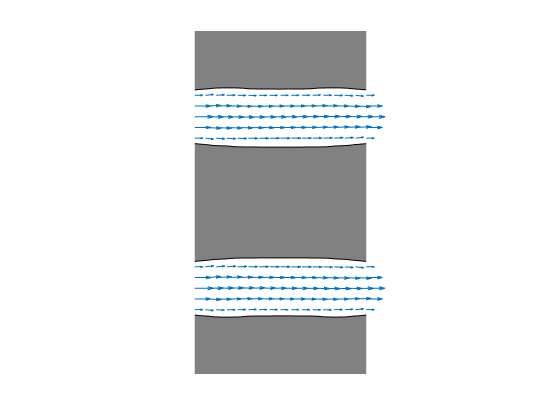}
\includegraphics[width=6cm,height=5.4cm,clip,trim=1cm 0cm 1cm 0cm]{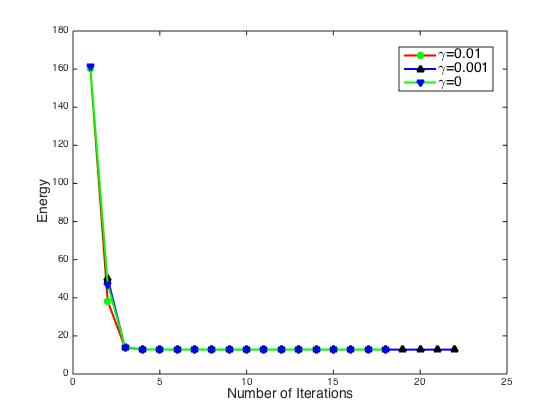}
\end{center}
\caption{\footnotesize (Example 6.2) For the case $d=0.5$. Left: Initial distribution of $\chi_1$. Middle: Optimal double pipes and approximate velocity in the fluid region. Right: For fixed $\tau=0.001$, energy curves for the cases of $\gamma=0.01,0.001,0$.}\label{ex2f2_1}
\end{figure}

For the case $d=0.5$, we choose a random initial distribution $\chi_1$, as shown in the left graph of Figure \ref{ex2f2_1}. We remark that $\gamma$ can also be set to zero in Algorithm \ref{a:MBO}. For fixed $\tau=0.001$, we test $\gamma=0.01,0.001,0$. The optimal design result is nearly the same for the three choices of $\gamma$, as shown in the middle graph of Figure \ref{ex2f2_1}, and the energy decaying property is observed from the energy curves in the right graph of Figure \ref{ex2f2_1}.

\begin{figure}[htbp]
\begin{center}
\includegraphics[width=8cm,height=5.4cm,clip,trim=1cm 2cm 1cm 1.5cm]{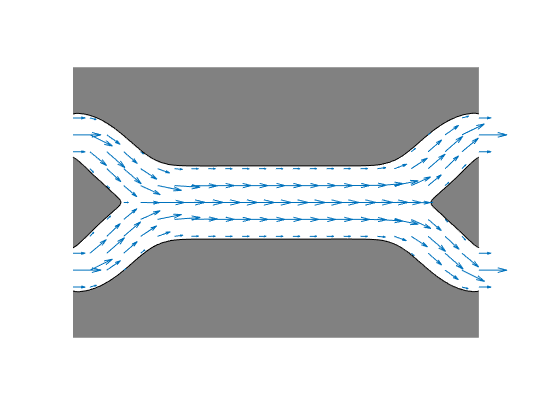}
\includegraphics[width=7cm,height=5.4cm,clip,trim=1cm 0cm 0cm 0cm]{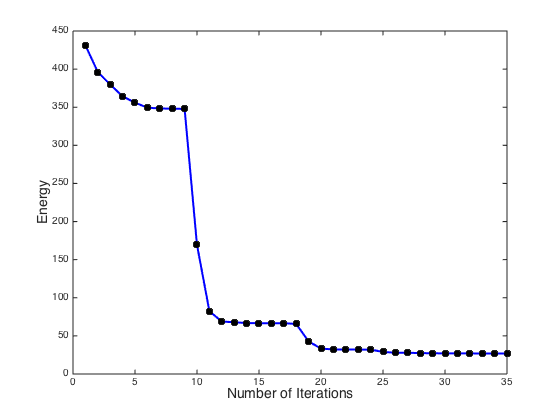}
\end{center}
\caption{\footnotesize (Example 6.2) For the case $d=1.5$, the parameters are set as $\tau=0.01$ and $\gamma=0.0001$. Left: Optimal double pipes and approximate velocity in fluid region. Right: Energy curve.}\label{ex2f2_2}
\end{figure}

For the case $d=1.5$, we choose an initial distribution $\chi_1$ with the fluid region located in the middle of the domain as Case 1 of Example 6.1. We set $\tau=0.01$ and $\gamma=0.0001$. The optimal design result and the approximate velocity are shown in the left graph of Figure \ref{ex2f2_2}, and the energy decaying property is also observed from the energy curve in the right graph of Figure \ref{ex2f2_2}. Compared with the computational cost used by MMA in \cite{Borrvall2003}, we find that our algorithm converges more quickly to the optimal result (cf. Table \ref{table2}).

\subsection*{Example 6.3}
We consider another  example studied  in \cite{Borrvall2003} that  includes a body fluid force term imposed in the local circular region with center $[1/2,1/3]$ and radius $r=1/12$. We show the design domain in Figure \ref{ex3f1}. The inflow and outflow Dirichlet boundaries are located with centers $[0,2/3]$ and $[1,2/3]$ respectively. Let $\overline{g} =1 $ for the inflow and outflow velocities, and let the fluid region fraction be $\beta=1/4$. We test the problem with various choices of body fluid force on a $128 \times 128$ grid, and we always choose $\bar{\alpha}=2.5\times10^4$, $\tau=0.01$, $\gamma=0.0001$ in this example.

\begin{figure}[htbp]
\begin{center}
\includegraphics[width=8.5cm,height=5.4cm,clip,trim=0cm 0cm 0cm 1cm]{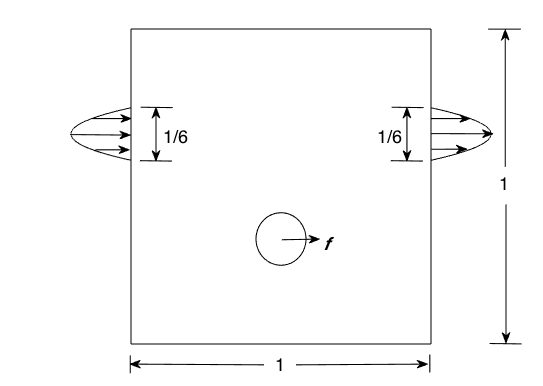}
\end{center}
\caption{\footnotesize (Example 6.3) Design domain for the example with a force term.}\label{ex3f1}
\end{figure}

\begin{figure}[htbp]
\begin{center}
\includegraphics[width=7cm,height=5.4cm,clip,trim=3cm 2cm 2cm 1cm]{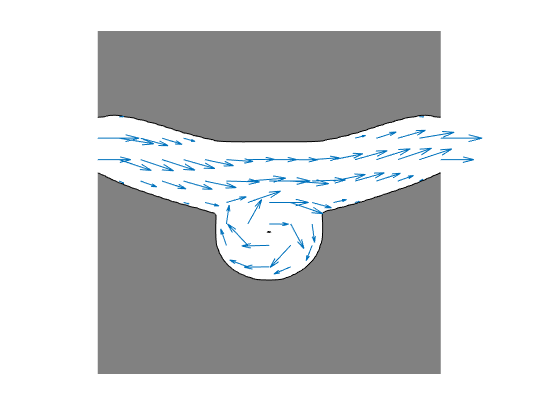}
\includegraphics[width=7cm,height=5.4cm,clip,trim=0cm 0.5cm 1cm 0.5cm]{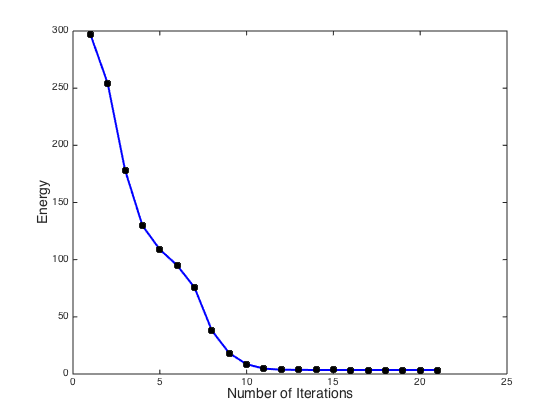}
\end{center}
\caption{\footnotesize (Example 6.3) For the example with a force term $\bff = [-1125,0]$ on a grid $128\times 128$. Left: Optimal design result and approximate velocity in the fluid region. Right: Energy curve.}\label{ex3f2_1}
\end{figure}

\begin{figure}[htbp]
\begin{center}
\includegraphics[width=7cm,height=5.4cm,clip,trim=3cm 2cm 2cm 1cm]{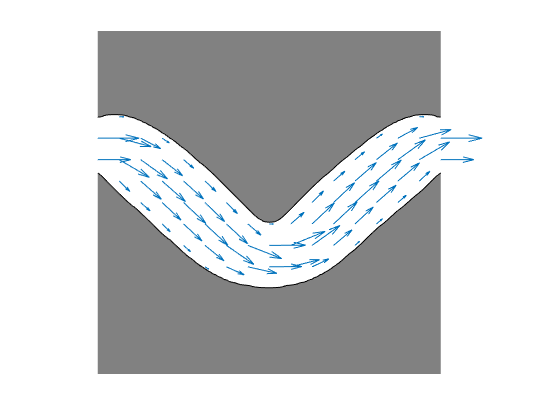}
\includegraphics[width=7cm,height=5.4cm,clip,trim=0cm 0.5cm 1cm 0.5cm]{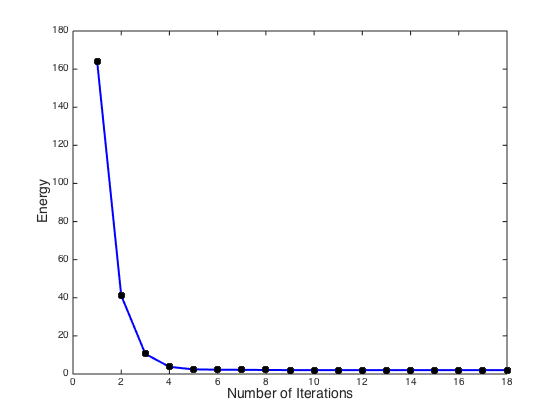}
\end{center}
\caption{\footnotesize (Example 6.3) For the example with a force term $ [562.5,0]$ on a grid $128\times 128$. Left: Optimal design result and approximate velocity in the fluid region. Right: Energy curve.}\label{ex3f2_2}
\end{figure}

\begin{figure}[htbp]
\begin{center}
\includegraphics[width=7cm,height=5.4cm,clip,trim=3cm 2cm 2cm 1cm]{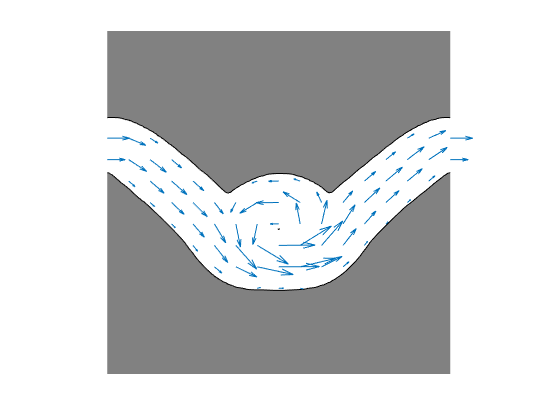}
\includegraphics[width=7cm,height=5.4cm,clip,trim=0cm 0.5cm 1cm 0.5cm]{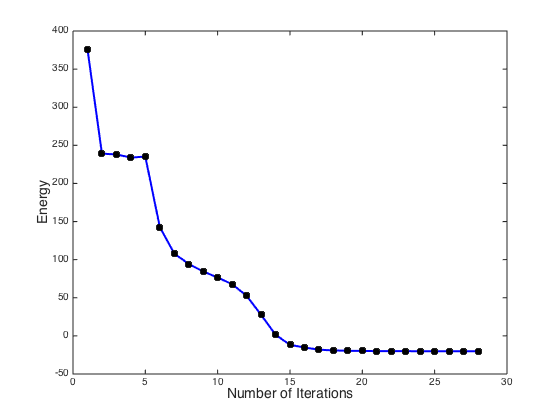}
\end{center}
\caption{\footnotesize (Example 6.3) For the example with a force term $ [1687.5,0]$ on a grid $128\times 128$. Left: Optimal design result and approximate velocity in the fluid region. Right: Energy curve.}\label{ex3f2_3}
\end{figure}

We test the cases for three different force terms $\bbf = [-1125,0],  [562.5,0], [1687.5,0]$. We choose the initial distribution $\chi_1$ with the fluid region located in a circular region with center $[1/2,1/2]$ and radius $1/\sqrt{3\pi}$. The optimal results and energy curves are plotted in Figures \ref{ex3f2_1} to \ref{ex3f2_3}  for various values for force $\bbf$, and the new algorithm also converges more quickly to the optimal results than the MMA shown in \cite{Borrvall2003}. One can observe that  for  $\bbf = [-1125,0]$  the fluid flow is  in a clockwise direction near the center roundabout (left graphs in Figure \ref{ex3f2_1}), while for $\bbf = [1687.5,0]$   it  is in a counterclockwise direction (left graph of  Figure \ref{ex3f2_3}).

\begin{figure}[htbp]
\begin{center}
\includegraphics[width=7cm,height=5.4cm,clip,trim=3cm 2cm 2cm 1cm]{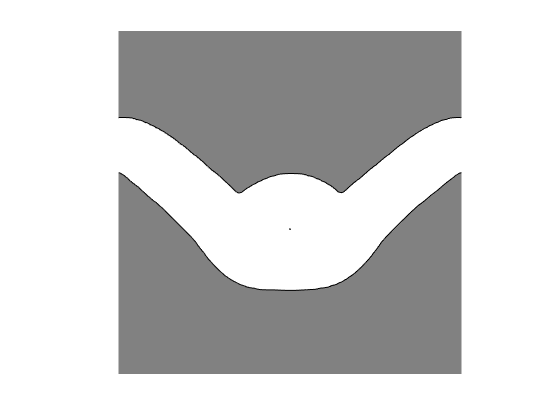}
\includegraphics[width=7cm,height=5.4cm,clip,trim=3cm 2cm 1cm 1cm]{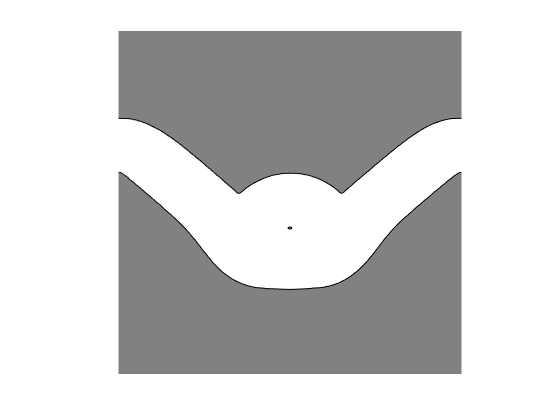}
\end{center}
\caption{\footnotesize (Example 6.3) Optimal design results for for example with force term $\bff = [1687.5,0] $. Left: Optimal design result on a coarse grid $128\times 128$. Right: Optimal design result on a fine grid $256\times 256$.}\label{ex3f3}
\end{figure}

An interesting phenomenon observed in this example was the appearance of a tiny local solid at the center of the roundabout for the two cases of $\bbf = [-1125,0], [1687.5,0]$, and the tiny local solid is clearer when the grid is finer (cf. Figure \ref{ex3f3}).

\subsection*{Example 6.4}
Finally, we consider optimal design for a three-terminal device shown in Figure \ref{ex4f1}. The inflow and outflow Dirichlet boundaries are located with centers $[0,0.3]$ and $[1,0.7]$, and the homogeneous Neumann boundary is located on the left boundary with center $[0,1.1]$. Let $\overline{g} =0.5$ for the inflow velocity and the fluid region fraction be $\beta=0.3$. We choose $\bar{\alpha}=2.5\times10^4$, $\tau=0.01$, $\gamma=0.0001$ in this example and test the problem on a grid $80\times 112$.

\begin{figure}[htbp]
\begin{center}
\includegraphics[width=8cm,height=6.2cm,clip,trim=0cm 0cm 0cm 1cm]{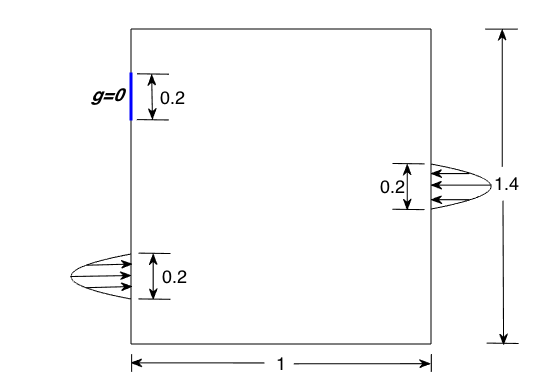}
\end{center}
\caption{\footnotesize (Example 6.4) Design domain for the example with a force term.}\label{ex4f1}
\end{figure}

\begin{figure}[htbp]
\begin{center}
\includegraphics[width=7.5cm,height=6cm,clip,trim=3cm 1cm 3cm 1cm]{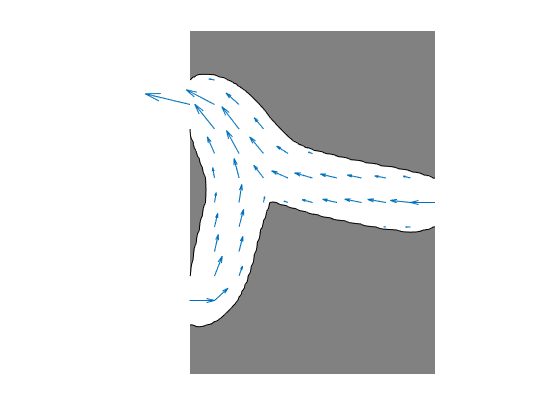}
\includegraphics[width=7.5cm,height=6cm,clip,trim=0cm 0cm 0cm 0.5cm]{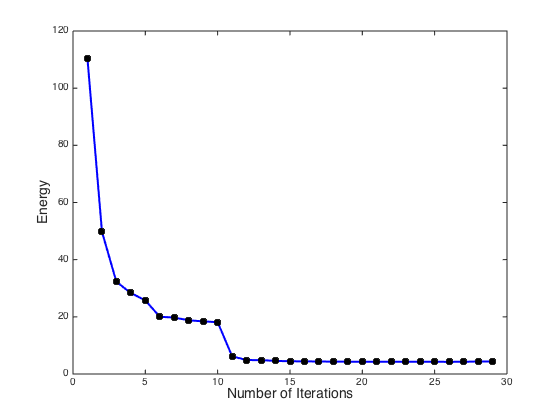}
\end{center}
\caption{\footnotesize (Example 6.4) Left: Optimal design result for example of three-terminal device and approximate velocity in the fluid region. Right: Energy curve.}\label{ex4f2}
\end{figure}

We choose the initial distribution $\chi_1$, with the  fluid region located in double parallel pipes $[0,1]\times[13/60,23/60] \cup [0,1]\times [37/60,47/60]$. The optimal result was obtained after 29 iterations. The optimal design result and the approximate velocity are shown in the left graph of Figure \ref{ex4f2}. The energy decaying property is also observed from the energy curve in the right graph of Figure \ref{ex4f2}.

\subsection{Three dimensional results}
We now present the  numerical examples in three dimensions. For the Dirichlet boundary condition in the following examples, we always assume that the magnitude of the velocity is set as
\begin{align*}
|\bu_D|=\bar{g}\Big(1-\frac{(y-a)^2+(z-b)^2}{l^2}\Big),
\end{align*}
where $\bar{g}$ is the prescribed velocity at the center of the flow profile at which the inflow/outflow velocity is imposed, $l$ is the radius of the flow profile, $(y,z)$ are Cartesian coordinates on a $x$-plane, and $(a,b)$ are the center of a circle on a $x$-plane.
\subsection*{Example 6.5.} The design domain of this example is shown in Figure \ref{ex5f1}. For the inflow, we let $\bar{g}=1$, $l=\frac{1}{2}$, and $(a,b)=(\frac{1}{2},\frac{1}{2})$ on $x=0$ plane. For the objective of mass conservation, we let $\bar{g}=9$, $l=\frac{1}{6}$, and $(a,b)=(\frac{1}{2},\frac{1}{2})$ on $x=1$ plane. We set the fluid region fraction is $\beta=0.35$. This example was already tested by the level set method in \cite{Challis2009}. Here we apply our new Algorithm \ref{algorithm_main} to obtain the optimal diffuser. Throughout this example, we choose the initial distribution $\chi_1$ with fluid domain in a region of $\{(x,y,z): x \in (0,1), y \in (0,1), z \in (\frac{7}{20},\frac{7}{10})\}$.

\begin{figure}[htbp]
\begin{center}
\includegraphics[width=7.5cm,height=6cm,clip,trim=0cm 0cm 0cm 1cm]{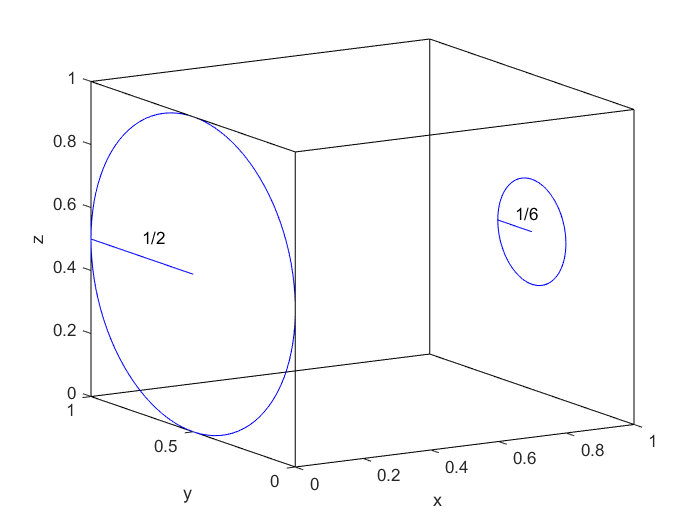}
\end{center}
\caption{\footnotesize (Example 6.5) Design domain.}\label{ex5f1}
\end{figure}

\begin{figure}[htbp]
\begin{center}
\includegraphics[width=8cm,height=6cm,clip,trim=0cm 0cm 0.5cm 1cm]{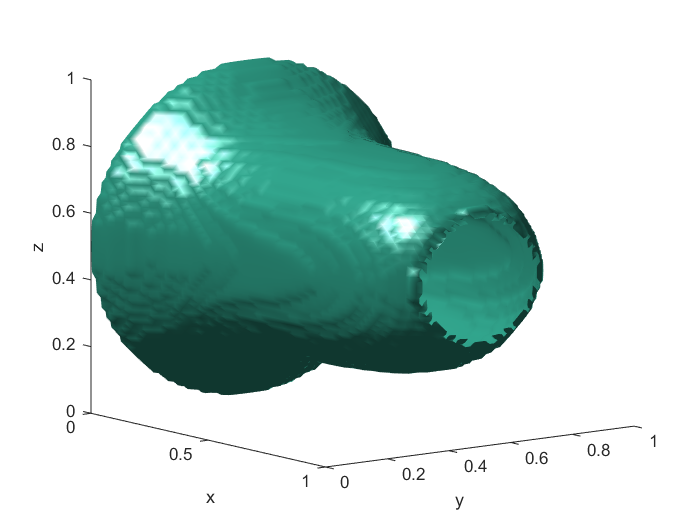}
\includegraphics[width=6.5cm,height=5.5cm,clip,trim=0cm 0cm 0cm 0.5cm]{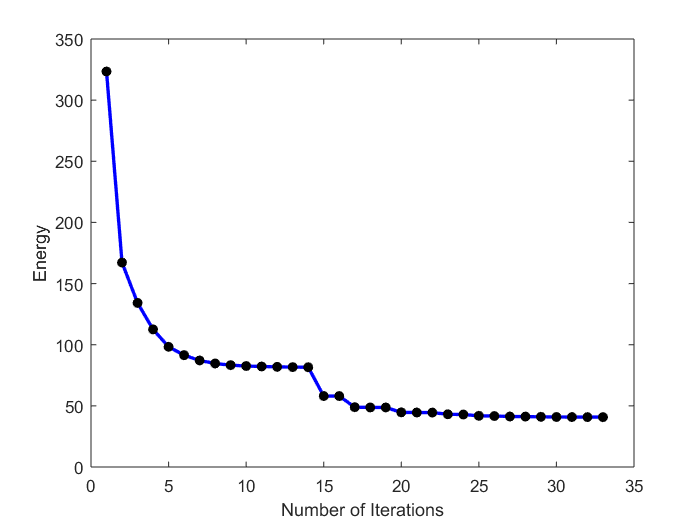}
\end{center}
\caption{\footnotesize (Example 6.5) Left: Optimal design result on a $32\times32\times32$ grid. Right: Energy curve.
In this case the parameters are set as $\bar{\alpha}=2.5\times10^4$, $\tau=0.05$, $\gamma=0.01$.}\label{ex5f2}
\end{figure}

\begin{figure}[htbp]
\begin{center}
\includegraphics[width=8cm,height=6cm,clip,trim=0cm 0cm 0.5cm 1cm]{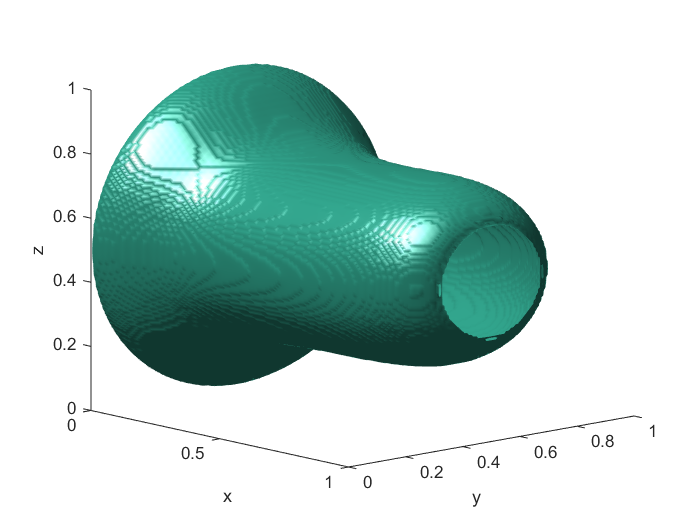}
\includegraphics[width=6.5cm,height=5.5cm,clip,trim=0cm 0cm 0cm 0.5cm]{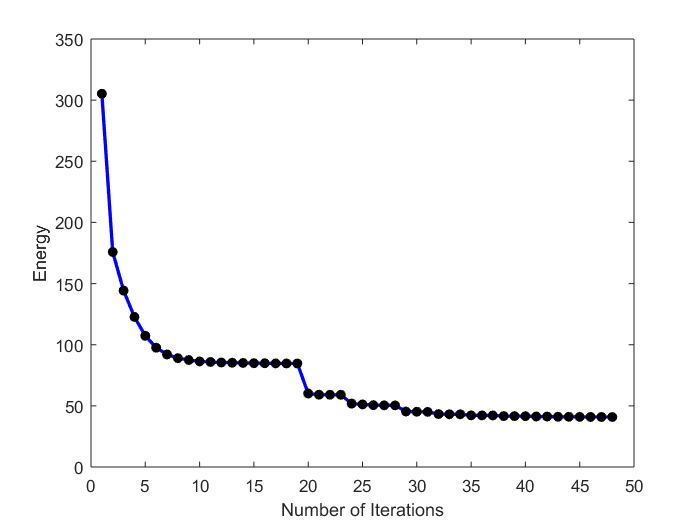}
\end{center}
\caption{\footnotesize (Example 6.5) Left: Optimal design result on a $64\times64\times64$ grid. Right: Energy curve.
In this case the parameters are set as $\bar{\alpha}=2.5\times10^4$, $\tau=0.05$, $\gamma=0.01$.}\label{ex5f3}
\end{figure}

Firstly, we test the case with $\bar{\alpha}=2.5\times10^4$, $\tau=0.05$, and $\gamma=0.01$ on $32\times32\times32$ and $64\times64\times64$ grids. In the following, the interface between solid and fluid regions for the optimal design is shown, and the fluid region locates in the interior of subdomain surrounded by the interface. The optimal diffusers are presented in the left graphs of Figure \ref{ex5f2} and Figure \ref{ex5f3} and the energy decaying property can be observed in the right graphs of Figure \ref{ex5f2} and Figure \ref{ex5f3}. The optimal design results seem to be similar to that in \cite{Challis2009}. The iteration converges in about 25 steps and 35 steps on coarse and fine grids respectively. Additionally, the slice of optimal design result at $y=0.5$ on $32\times32\times32$ grid and the approximate velocity in the fluid domain are provided in Figure \ref{ex5f5}.

\begin{figure}[htbp]
\begin{center}
\includegraphics[width=9cm,height=6.5cm,clip,trim=0cm 0cm 0cm 0cm]{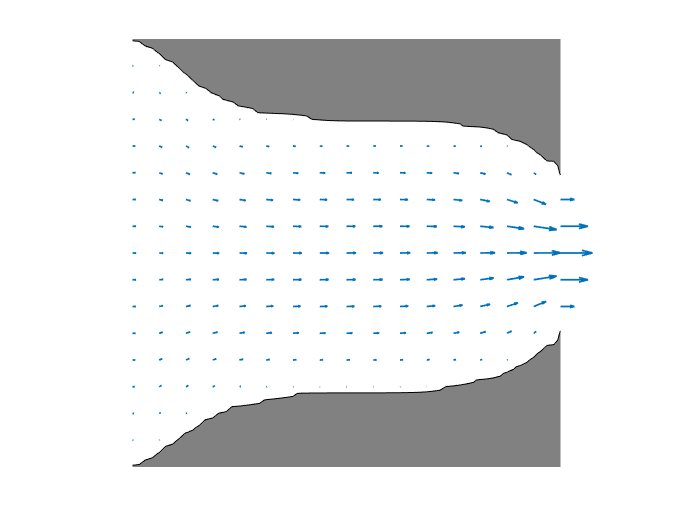}
\end{center}
\caption{\footnotesize (Example 6.5) The slice of optimal design result and the approximate velocity in fluid region at $y=0.5$ on a $32\times32\times32$ grid. The parameters are set as $\bar{\alpha}=2.5\times10^4$, $\tau=0.05$, $\gamma=0.01$.}\label{ex5f5}
\end{figure}

Next, the energy decay properties of the Algorithm \ref{algorithm_main} with different parameters $\tau$ and $\gamma$ for this problem are shown for the same case of $\bar{\alpha}=2.5\times10^4$ in Figure \ref{ex5f4}. We note that the optimal design results for different parameters $\tau$ and $\gamma$ are similar to that in the left graphs of Figure \ref{ex5f2} and Figure \ref{ex5f3}. From the two graphs of Figure \ref{ex5f4}, we find that the energy converges to almost the same value when $\tau$ or $\gamma$ is fixed.

\begin{figure}[htbp]
\begin{center}
\includegraphics[width=7cm,height=5.4cm,clip,trim=0cm 0cm 0cm 0.5cm]{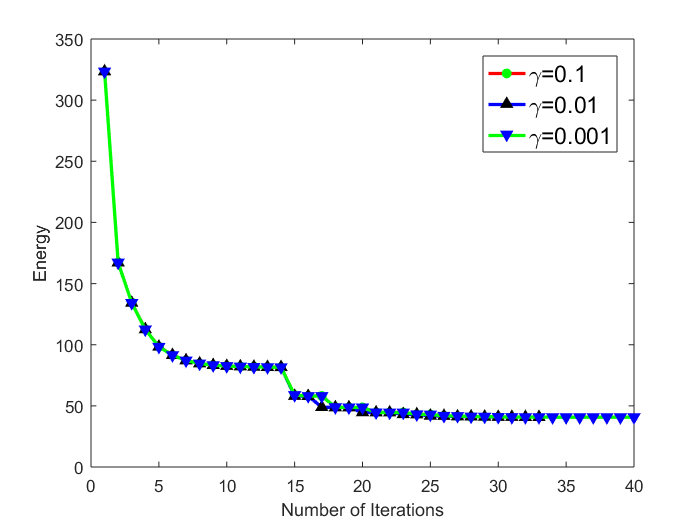}
\includegraphics[width=7cm,height=5.4cm,clip,trim=0cm 0cm 0cm 0.5cm]{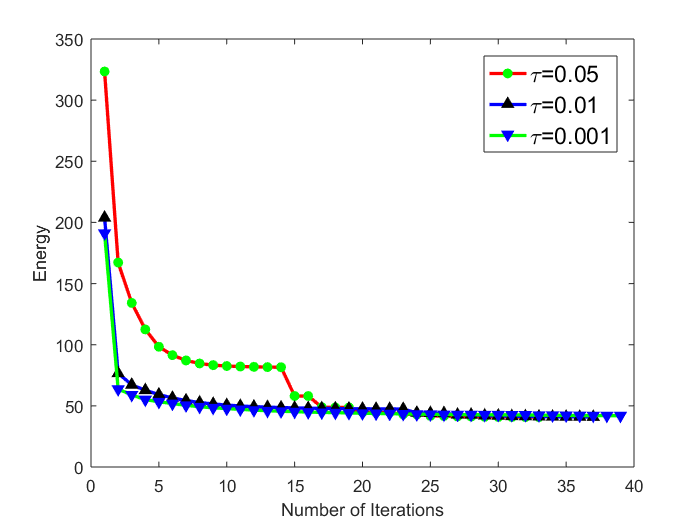}
\end{center}
\caption{\footnotesize (Example 6.5) Plot of energy curves for $\bar{\alpha}=2.5\times10^4$ on $32\times32\times32$ grid. Left: For fixed $\tau=0.05$, energy curves for the cases of $\gamma = 0.1,0.01,0.001$. Right: For fixed $\gamma=0.01$, energy curves for the cases of $\tau = 0.05,0.01,0.001$.}\label{ex5f4}
\end{figure}

\subsection*{Example 6.6.} In this example we assume that there are four flow profiles on the inflow boundary and one flow profile on the outflow boundary. The design domain is shown in Figure \ref{ex6f1}. For the four inflow profiles, we let $\bar{g}=1$, the radius is set as $l=\frac{1}{8}$ and the centers of circles are $(\frac{1}{4},\frac{1}{4})$, $(\frac{1}{4},\frac{3}{4})$, $(\frac{3}{4},\frac{1}{4})$ and $(\frac{3}{4},\frac{3}{4})$ on the $x=0$ plane respectively. For the outflow profile, we let $\bar{g}=1$, $l=\frac{1}{4}$ and $(a,b)=(\frac{1}{2},\frac{1}{2})$ on the $x=1$ plane. We set the fluid region fraction as $\beta=\frac{1}{4}$.

\begin{figure}[htbp]
\begin{center}
\includegraphics[width=7.5cm,height=6cm,clip,trim=0cm 0cm 0cm 1cm]{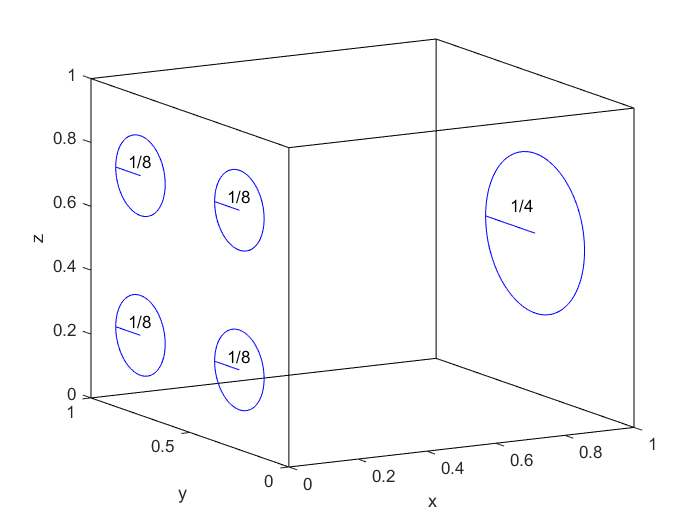}
\end{center}
\caption{\footnotesize (Example 6.6) Design domain.}\label{ex6f1}
\end{figure}

We test this problem based on the Algorithm \ref{algorithm_main} with $\bar{\alpha}=2.5\times10^4$, $\tau=0.05$, and $\gamma=0.01$ on $32\times32\times32$ and $64\times64\times64$ grids. The initial distribution $\chi_1$ with fluid domain is located in a region of $\{(x,y,z): x\in (0,1),y\in(0,1),z \in (\frac{1}{2},\frac{3}{4})\}$. The corresponding optimal design result is shown in the left graphs of Figure \ref{ex6f2} and Figure \ref{ex6f3}. From the left graphs of Figure \ref{ex6f2} and Figure \ref{ex6f3}, we can see that the interface between solid and fluid regions is more smooth when the simulation is performed on the fine grid. From the right graphs of Figure \ref{ex6f2} and Figure \ref{ex6f3}, the energy decaying property is also observed. The iteration converges in about 50 steps and 70 steps on coarse and fine grids respectively. In Figure \ref{ex6f4}, we present the slice of optimal design result at $z=25/64$ on a $32\times32\times32$ grid, and the approximate velocity in the fluid region is also included.

\begin{figure}[htbp]
\begin{center}
\includegraphics[width=8cm,height=6.3cm,clip,trim=0cm 0cm 0.5cm 1cm]{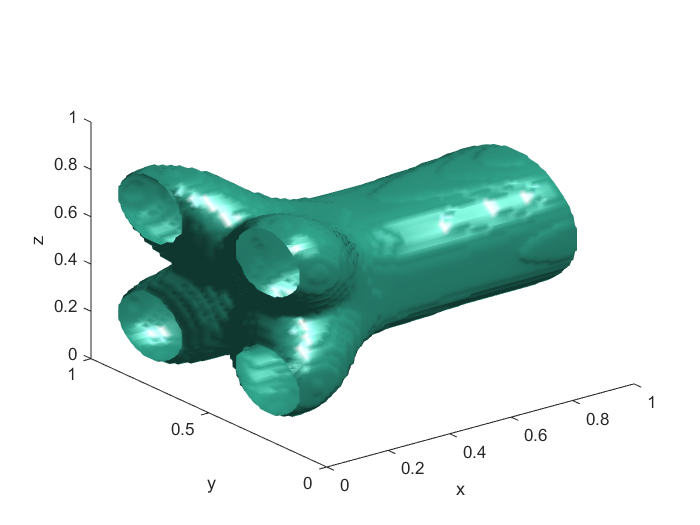}
\includegraphics[width=6.5cm,height=5.3cm,clip,trim=0cm 0cm 0cm 0.5cm]{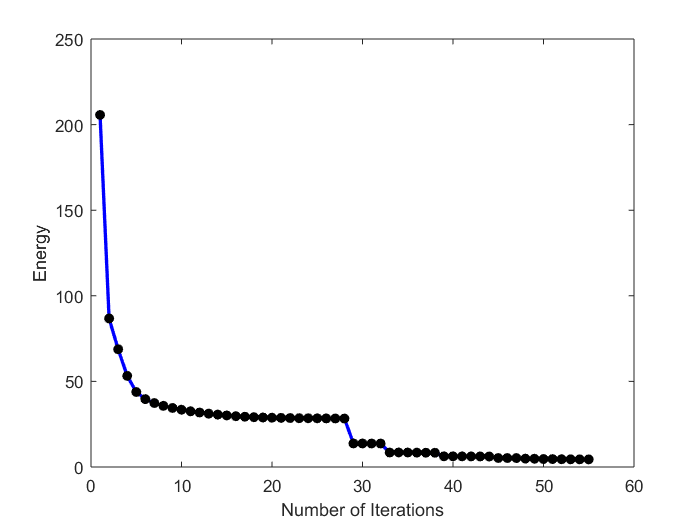}
\end{center}
\caption{\footnotesize (Example 6.6) Left: Optimal design result on a $32\times32\times32$ grid. Right: Energy curve.
In this case the parameters are set as $\bar{\alpha}=2.5\times10^4$, $\tau=0.05$, $\gamma=0.01$.}\label{ex6f2}
\end{figure}

\begin{figure}[htbp]
\begin{center}
\includegraphics[width=8cm,height=6.3cm,clip,trim=0cm 0cm 0.5cm 1cm]{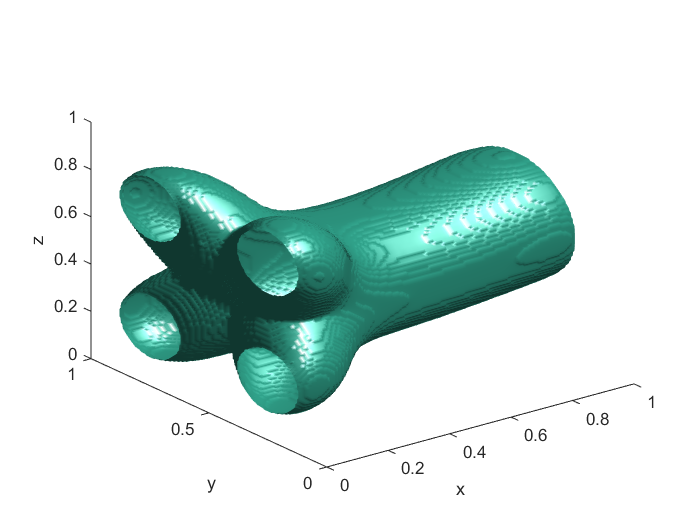}
\includegraphics[width=6.5cm,height=5.3cm,clip,trim=0cm 0cm 0cm 0.5cm]{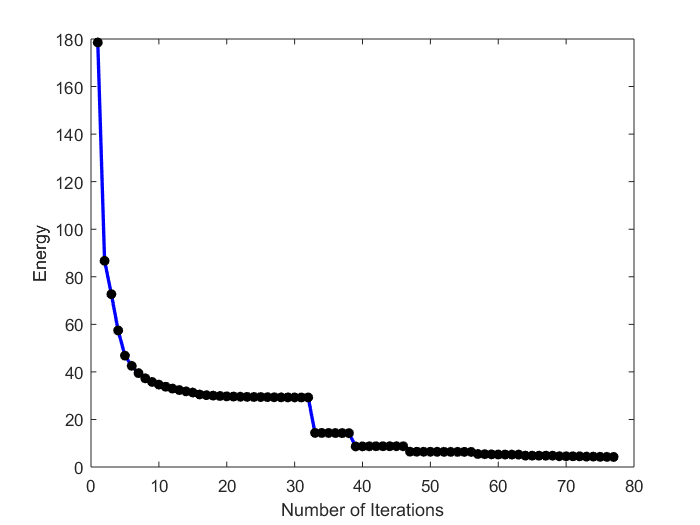}
\end{center}
\caption{\footnotesize (Example 6.6) Left: Optimal design result on a $64\times64\times64$ grid. Right: Energy curve.
In this case the parameters are set as $\bar{\alpha}=2.5\times10^4$, $\tau=0.05$, $\gamma=0.01$.}\label{ex6f3}
\end{figure}

\begin{figure}[htbp]
\begin{center}
\includegraphics[width=9cm,height=5.6cm,clip,trim=0cm 1cm 0cm 0.5cm]{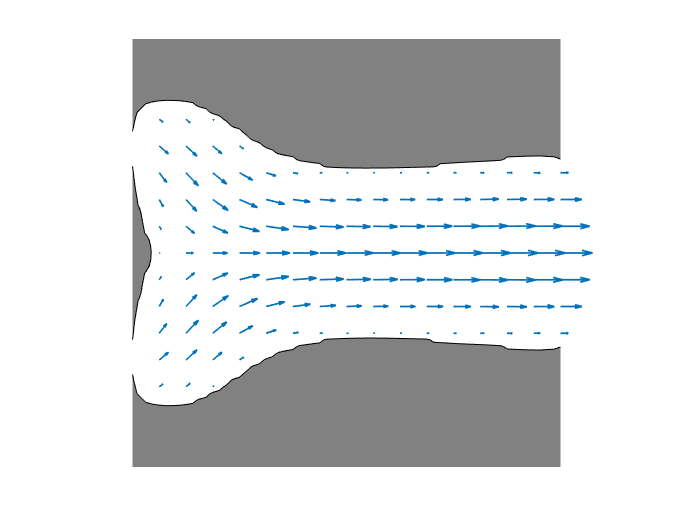}
\end{center}
\caption{\footnotesize (Example 6.6) The slice of optimal design result and the approximate velocity in fluid region at $z=25/64$ on a $32\times32\times32$ grid. The parameters are set as $\bar{\alpha}=2.5\times10^4$, $\tau=0.05$, $\gamma=0.01$.}\label{ex6f4}
\end{figure}

\subsection{Discussions on the robustness and efficiency of our algorithm}
The  numerical results in the previous subsections demonstrated the robustness and efficiency of our algorithm. First,  the final optimal design result seems to be  insensitive to the initial distribution of $\chi_1$. As shown in the first and second two dimensional examples for the case with $\overline{\alpha} = 2.5\times 10^4$, even with a random initial distribution of $\chi_1$, we always get the same final optimal diffuser (cf. Figures \ref{ex1f2_dis}-\ref{ex1f2_energ_vel_0}).   From the viewpoint of energy stability, the energy decaying property is proved mathematically and observed numerically for the problem with different initial distributions of $\chi_1$.  Moreover, from the numerical results in Figure \ref{ex1f2_diff_pram} for the Example 6.1 and Figure \ref{ex5f4} for the Example 6.5, we can see that our algorithm is also robust for the different choices of parameters used in the algorithm.

Next, we compare some of the numerical results above with some existing methods for topology optimization of fluids in Stokes flow in the literature. We mainly compare the numerical results of our algorithm with the results using  MMA in \cite{Borrvall2003} and the level set approach in \cite{Challis2009}.

\begin{table}[ht!]\renewcommand{\arraystretch}{1.2}
\footnotesize
\begin{tabular}{c|c|c }
Methods & Grid & Number of iterations \\ \hline\hline
MMA & $100\times 100$  & 33   \\
Level set & $96\times 96$  & 197  \\
Our algorithm & $128\times 128$  & 21\\
\hline
 \end{tabular}
\smallskip
\smallskip
\caption{\footnotesize Comparison of the number of iterations of different methods to obtain the optimal design result for Example 6.1 with $\overline{\alpha} = 2.5\times 10^4$. The parameters used in our algorithm are set as $\tau = 0.001$, $\gamma=0.01$.}\label{table1}
\end{table}

 \begin{table}[ht!]\renewcommand{\arraystretch}{1.2}
\footnotesize
\begin{tabular}{c|c|c }
Methods & Grid & Number of iterations \\ \hline\hline
MMA & $150\times 100$  & 236   \\
Level set & $216\times 144$  & 681  \\
Our algorithm & $192\times 128$  & 35\\
\hline
 \end{tabular}
\smallskip
\smallskip
\caption{\footnotesize Comparison of the number of iterations of different methods to obtain the optimal design result for Example 6.2 with $d= 1.5$. The parameters used in our algorithm are set as $\tau = 0.01$, $\gamma=0.0001$.}\label{table2}
\end{table}

 \begin{table}[ht!]\renewcommand{\arraystretch}{1.2}
\footnotesize
\begin{tabular}{c|c|c }
Methods & Surface force density & Number of iterations \\ \hline\hline
 &  $-1125$  & 229   \\
MMA & $562.5$    & 66   \\
 & $1687.5$    & 69\\
 \hline
  &  $-1125$  & 21   \\
Our algorithm & $562.5$    & 18   \\
 & $1687.5$    & 28\\
\hline
 \end{tabular}
\smallskip
\smallskip
\caption{\footnotesize Comparison of the number of iterations of the MMA and our algorithm to obtain the optimal design results for Example 6.3. The MMA is tested on a grid $100\times 100$, and our algorithm is tested on a grid $128\times 128$.}\label{table3}
\end{table}


\begin{table}[ht!]\renewcommand{\arraystretch}{1.2}
\footnotesize
\begin{tabular}{c|c|c }
\rm Methods &\rm Grid &{\rm Number of iterations}\\
\hline\hline
Level set & $36\times36\times36$& 316  \\
 &$60\times60\times60 $& 647  \\
\hline
Our algorithm &$32\times32\times32$ & 33    \\
 &$64\times64\times64$ & 48  \\
 \hline
 \end{tabular}
\smallskip
\smallskip
\caption{\footnotesize Comparison of different methods for Example 6.5. The parameters used in our algorithm are set as $\bar{\alpha}=2.5\times10^4$, $\tau = 0.05$, $\gamma=0.01$.}\label{table4}
\end{table}

In our algorithm, only a Brinkman problem is solved without the need to solve adjoint problem at each iteration step, and the indicator functions of fluid-solid regions are easily updated based on simple convolutions followed by a thresholding step. Therefore, the computational cost at each iteration is less than that in MMA \cite{Borrvall2003} or in the level set approach \cite{Challis2009}. Thus, our algorithm is much simpler and easier to implement than those methods. Tables \ref{table1} and \ref{table2} list the number of iterations of our algorithm, the MMA, and the level set approach for two examples, Table \ref{table3} shows the number of iterations of the MMA and our algorithm for Example 6.3, and Table \ref{table4} presents the number of iterations of the level set approach and our algorithm for Example 6.5. We can see that our algorithm converges in many fewer steps.

\section{Discussion and conclusions}\label{sec:conclusion}
In this paper, we introduce a new efficient threshold dynamics method for topology optimization for fluids in Stokes flow. We aim to minimize a total energy functional that consists of the dissipation power and the perimeter approximated by nonlocal energy. During the iterations of the algorithm, only a Brinkman equation requires solution by a mixed finite-element method, and the indicator functions of fluid-solid regions are updated by a thresholding step that is based on the convolutions computed by the FFT. A simple adaptive in time strategy is used to accelerate the convergence of the algorithm. The total energy decaying property of the proposed algorithm is rigorously proved and observed numerically. Several numerical examples are tested to verify the efficiency of the new algorithm, and we show that the new algorithm converges more rapidly for most the examples than the MMA used in \cite{Borrvall2003}.  Compared to existing methods for topology optimization for fluids, we believe that the proposed algorithm is simple and easy to implement.   For the numerical experiments that we have performed thus far, the proposed method always finds an optimal topology and the numerical results are insensitive to the initial guess and parameters. We believe that our algorithm can also be extended to topology optimization for fluids in Navier-Stokes flow.

\providecommand{\href}[2]{#2}

\end{document}